\documentclass[11pt,draft]{amsart}
\usepackage{dsfont}
\usepackage{amsfonts}
\usepackage{mathrsfs}
\usepackage{amssymb}
\usepackage{amsmath}
\usepackage{amsfonts}
\usepackage{amsfonts,enumerate}
\usepackage{pifont}
\usepackage{enumerate}
\usepackage{graphics}
\usepackage{verbatim}
\usepackage{amsthm}
\usepackage{latexsym,bm}
\usepackage{color}
\numberwithin{equation}{section}
\newtheorem{theorem}{Theorem}[section]
\newtheorem{corollary}[theorem]{Corollary}
\newtheorem{lemma}[theorem]{Lemma}

\theoremstyle{definition}
\newtheorem{remark}[theorem]{Remark}
\newtheorem{definition}[theorem]{Definition}

\newcommand{\E}{{\mathcal E}}

\newcommand{\M}{{\mathcal M}}

\newcommand{\BMO}{{\mathcal {BMO}}}

\newcommand{\be}{\begin{eqnarray*}}
\newcommand{\ee}{\end{eqnarray*}}
\newcommand{\beq}{\begin{equation}}
\newcommand{\eeq}{\end{equation}}
\newcommand{\beqn}{\begin{equation*}}
\newcommand{\eeqn}{\end{equation*}}

\begin{document}

\title[]{John-Nirenerg inequality and atomic decomposition for noncommutative
martingales}

\author{Guixiang Hong$^*$}
\address{Laboratoire de Math{\'e}matiques, Universit{\'e} de Franche-Comt{\'e}, \newline
25030 Besan\c{c}on Cedex, France\\ \emph{E-mail address:
guixiang.hong@univ-fcomte.fr}}
\thanks{*\  Partially supported by ANR grant 2011-BS01-008-01.}

\author{Tao Mei$^\dag$}
\address{Department of Mathematics, Wayne State University, \newline 656 W. Kirby
Detroit, MI 48202. USA \\ \emph{E-mail address: mei@wayne.edu}}
\thanks{\dag\  Partially supported by NSF grant DMS-0901009.}

\thanks{\small {{\it MR(2000) Subject Classification}.} Primary
46L52, 46L53; Secondary 60G42, 60G46.}
.

\maketitle

\begin{abstract}
In this paper, we study the John-Nirenberg inequality for $\BMO$ and
the atomic decomposition for $\mathcal{H}_1$ of noncommutative
martingales. We first establish a crude version of the column (resp.
row) John-Nirenberg inequality for all $0<p<\infty$. By an extreme
point property of $L_p$-space for $0<p\leq1$, we then obtain a fine
version of this inequality.  The  latter corresponds exactly  to the
classical John-Nirenberg inequality and enables us to obtain an
exponential integrability inequality like in the classical case.
These results extend and improve  Junge and Musat's John-Nirenberg
inequality. By duality, we obtain the corresponding $q$-atomic
decomposition for different Hardy spaces $\mathcal{H}_1$  for all
$1<q\leq\infty$, which extends the $2$-atomic decomposition
previously obtained by Bekjan {\it et al}. Finally, we give a
negative answer to  a question posed by Junge and Musat about
$\BMO$.
\end{abstract}

\section{Introduction}

This paper deals with BMO spaces and atomic decomposition for
noncommutative martingales. The modern period of development of
noncommutative martingale inequalities began with Pisier and Xu's
seminal paper \cite{PiXu97} in which they established the
noncommutative Burkholder-Gundy inequalities and Fefferman duality
theorem between $\mathcal{H}_1$ and $\BMO$. Since then remarkable
progress has been made in the field. We refer, for instance, to
\cite{Jun02}, \cite{JuXu03}, \cite{JuXu06}, \cite{Ran02} for other
noncommutative martingales inequalities, to \cite{Mus03},
\cite{BCPY} for interpolation of noncommutative Hardy spaces and to
\cite{PaNa06}, \cite{Per09} for the noncommutative Gundy and Davis
decompositions. Let us also mention two other works that motivate
the present paper. The first one is Junge and Musat's noncommutative
John-Nirenberg theorem \cite{JuMu07} and the second the $2$-atomic
decomposition of the Hardy spaces $\mathcal{H}_1$ by Bekjan, Chen,
Perrin and Yin \cite{BCPY}.

Before describing our main results, we recall the classical John-Nirenberg
inequalities in the martingale theory. Let
$(\Omega, \mathcal{F}, \mathbb{P})$ be a probability space and
$(\mathcal{F}_n)_{n\geq0}$ an increasing sequence of
sub-$\sigma$-algebras of $\mathcal{F}$ with the associated
conditional expectations $(\mathbb E_n)_{n\geq0}$. The $BMO(\Omega)$
space is defined as the set of all $x\in L_1(\Omega)$ with the norm
\begin{align}\label{commu-bmo}
\|x\|_{BMO}=\sup_n\|\mathbb E_n|x-x_{n-1}|\|_{\infty}<\infty.
\end{align}
The classical John-Nirenberg theorem says that there exist two
universal constants $c_1$, $c_2>0$ such that if $\|x\|_{BMO}< c_2$,
then
\begin{align}\label{commu-jn-exp}
\sup_{n}\|\mathbb{E}_n(e^{c_1|x-x_{n-1}|})\|_{\infty}<1.
\end{align}
This statement is equivalent to the following one: There exists an absolute constant $c$
such that for all $1\leq p<\infty$,
\begin{align}\label{commu-jn-1}
\|x\|_{BMO}\leq\sup_{n}\|\mathbb
E_n|x-x_{n-1}|^p\|^{\frac{1}{p}}_{\infty}\leq cp\|x\|_{BMO}.
\end{align}
 A duality argument yields
\begin{align}
\|\mathbb E_n|x-x_{n-1}|^{p}\|^{\frac 1p}_{\infty}&=\sup_{b\in
L_\infty(\mathcal{F}_n),\|b\|_1\leq1}\left(\int|x-x_{n-1}|^pbd\mathbb{P}\right)^{\frac 1p}\label{motivation-1}\\
&=\sup_{b\in
L_\infty(\mathcal{F}_n),\|b\|_p\leq1}\|(x-x_{n-1})b\|_p.\label{motivation-2}
\end{align}
Furthermore, by the extreme point property of $L_1(\mathcal{F}_n)$
and (\ref{motivation-1}), the John-Nirenberg theorem
(\ref{commu-jn-1}) can be rewritten as follows
\begin{align}\label{commu-jn-2}
\|x\|_{BMO}\leq\sup_{n}\sup_{E\in\mathcal{F}_n}\frac{1}{{\mathbb{P}(E)}^{1/p}}\|(x-x_{n-1})\mathds{1}_E\|_p\leq
cp\|x\|_{BMO}.
\end{align}
 Accordingly, \eqref{commu-jn-exp} can be reformulated as: For any $n\ge1$,  $E\in\mathcal{F}_n$ and $\lambda>0$
  \begin{align}\label{commu-jn-exp-2}
\frac{1}{\mathbb{P}(E)}\,\mathbb{P}\big(\big\{\omega\in E\,:\,
|x(\omega)-x_{n-1}(\omega)|>\lambda\big\}\big)\le
c_2\exp(-c_1\lambda/\|x\|_{BMO}).
\end{align}

Junge and Musat \cite{JuMu07} proved a noncommutative version of John-Nirenberg theorem corresponding to (\ref{motivation-2}). To state their result we need fix some notation. Let $\M$ be a von Neumann algebra  with a
normal faithful tracial state $\tau$.  Let
$(\M_n)_{n\geq1}$ be an increasing sequence of von Neumann
subalgebras of $\M$ such that the union of $\M_n$'s is $w^*$-dense
in $\M$. Let $\mathcal{E}_n$ be the conditional expectation of $\M$
with respect to $\M_n$.  Define
$$\|x\|_{\BMO^c}=\sup_{n\geq1}
\|\mathcal{E}_n|x-x_{n-1}|^2\|^{\frac 12}_{\infty}$$ and
 $$\BMO(\M)=\{x\in L_1(\M): \|x\|_{\BMO}<\infty\}$$
with
$$\|x\|_{\BMO}=\max\{\|x\|_{\BMO^c},\|x^*\|_{\BMO^c}\}.$$
 Then Junge and Musat's John-Nirenberg inequality reads as follows: There
exists an absolute constant $c$ such that for all $2\leq p<\infty$,
$$\|x\|_{\BMO}\leq \mathcal{B}_p(x)\leq cp\|x\|_{\BMO},$$
where
\begin{align*}
\mathcal{B}_p(x)=\max\{
&\sup_n\sup_{b\in\M_n,\|b\|_p\leq1} \|(x-x_{n-1})b\|_{p},\\
&\sup_n\sup_{b\in \M_n,\|b\|_p\leq1,} \|b(x-x_{n-1})\|_{p}\}.
\end{align*}

\renewcommand{\thetheorem}{\Alph{theorem}}

However, this theorem does not correspond to the commonly used form
of the classical John-Nirenberg inequality. On the other hand, it
does not hold (see Remark \ref{counter example2} for a
counterexample) when considering $\BMO^c(\M)$ or $\BMO^r(\M)$
separately.  The first purpose of this paper is to remedy these
aspects of Junge and Musat's theorem. The following is one of our
main results. We refer to the next section for all spaces and
notations used below. $\mathcal{P}(\M)$ denotes the set of all
projections of $\M$.

\begin{theorem}\label{fcb}
For $0<p<\infty$, we have
$$\alpha^{-1}_p\|x\|_{\BMO}\leq\mathcal{PB}_p(x)\leq\beta_p\|x\|_{\BMO},$$
where
\begin{align*}
\mathcal{PB}_p(x)=\max\{
&\sup_n\sup_{e\in \mathcal{P}(\M_n)} \|(x-x_{n-1})\frac{e}{(\tau(e))^{1/p}}\|_{p},\\
&\sup_n\sup_{e\in \mathcal{P}(\M_n)}
\|\frac{e}{(\tau(e))^{1/p}}(x-x_{n-1})\|_{p}\}.
\end{align*}
The two constants $\alpha_p$ and $\beta_p$ have the following
properties
 \begin{enumerate}[\rm (i)]

\item $\alpha_p=1$ for $2\leq p<\infty$;

\item $\alpha_p\leq C^{1/p-1/2}$ for $0<p<2$;

\item $\beta_p\leq cp$ for $2\leq p<\infty$;

\item  $\beta_p=1$ for $0<p<2.$
 \end{enumerate}
\end{theorem}

This result goes beyond Junge/Musat's result in two aspects. First
we extend their result to all $0<p<\infty$. Second, the $b$'s in the
definition of $\mathcal{B}_p(\cdot)$ are reduced to projections
$e$'s in $\mathcal{PB}_p(\cdot)$, which corresponds exactly to the
form (\ref{commu-jn-2}) in the classical case. Furthermore, the
optimal constants $\beta_p$ in Theorem A enable us to formulate
John-Nirenberg inequality that corresponds to the form
(\ref{commu-jn-exp-2}). That is, let $x\in \BMO(\M)$, then for all
natural numbers $n\geq1$, all $e\in \mathcal{P}(\M_n)$ and for all
$\lambda>0$, we have
$$\frac{1}{\tau(e)}\tau(\mathds{1}_{(\lambda,\infty)}(|(x-x_{n-1})e|)+\mathds{1}_{(\lambda,\infty)}(|e(x-x_{n-1})|))\leq 4\exp(-\frac{c\lambda}{\|x\|_{\BMO}})$$
with $c$ an absolute constant.

By the essentially same idea, we establish similar results for
$\BMO^c(\M)$ and $\BMO^r(\M)$ separately, but only with $2\leq
p<\infty$ (see Remark \ref{counter example1}).

We now turn to the second objective of this paper: the atomic
decomposition of different noncommutative Hardy spaces. Let us
recall the 2-atomic decomposition obtained in \cite{BCPY}. An
element $a\in L_1(\M)$ is said to be a $(1,2)_c$-atom with respect
to $(\M_n)_{n\geq1}$, if there exist $n\geq1$ and $e\in
\mathcal{P}(\M_n)$ such that

(i) $\E_n(a)=0$; \ (ii)$ae=a$; (iii) $\|a\|_2\leq(\tau( e))^{-1/2}$.

\noindent The atomic Hardy space $\mathsf h_{1,\rm{at}}^c(\M)$ is
defined as the space of all $x\in L_1(\M)$, such that the following
$\|\cdot\|_{\mathsf h_{1,\rm{at}}^c}$ norm is finite,
$$\|x\|_{\mathsf h_{1,\rm{at}}^c}=\|\E_1x\|_1+\inf\sum_j |\lambda_j|.$$
Here the infimum is taken for possible decompositions
$x-\E_1x=\sum_j \lambda_ja_j$ with $\lambda_j\in {\Bbb C}$,
$a_j$ being $(1,2)_c$-atom. It is proved in \cite{BCPY}
that $x\in \mathsf{h}_1^c(\M)$ if and only if $x\in
\mathsf{h}_{1,\rm{at}}^c(\M)$ and
$$\|x\|_{\mathsf h_1^c}\simeq \|x\|_{\mathsf h_{1,\rm{at}}^c}.$$
Together with the equivalence ${\mathcal H}_1^c(\M)=\mathsf
h_1^c(\M)+\mathsf h_1^d(\M)$, the authors of \cite{BCPY} also
obtained a $2$-atomic decomposition for ${\mathcal H}_1^c(\M)$.

Let us briefly recall the argument used in \cite{BCPY}. The dual
space of $\mathrm h_{1,\mathrm{at}}^c(\M)$ can be described as
$$\Lambda^c(\M)=\{x\in L_2(\M):\|x\|_{\Lambda^c}<\infty\}$$
with
$$\|x\|_{\Lambda^c}=\max\{\|\E_1x\|_\infty,\quad\sup_{n\geq1}\sup_{e\in\mathcal{P}_n}
(\frac1{\tau (e)}\tau(e|x-x_n|^2))^{\frac12}\}.$$ Actually, the
supremum in the definition above can be taken for all $b\in
L_1(\M_n)$ since the extreme points of the unit ball of
$L_1({\M}_n)$ are all multiples of projections. Therefore,
 \begin{eqnarray}
 \|x\|_{\Lambda^c}&=&\max\{\|\E_1x\|_\infty,\quad\sup_{n\geq1}\sup_{b\in \M_n}
(\frac1{\|b\|_1}\tau(b|x-x_n|^2))^{\frac12}\} \label{key}\\
&=&\max\{\|\E_1x\|_\infty,\quad\sup_{n\geq1}\|\E_n|x-x_n|^2\|_\infty^{\frac12}\}\nonumber\\
&=&\|x\|_{\mathsf {bmo}^c}.\nonumber
\end{eqnarray}
Then the duality
$\mathsf h_1^c(\M)=\mathsf {bmo}^c(\M)$ yields $\mathsf
h_{1,\rm{at}}^c(\M)=\mathsf h_1^c(\M)$.

It is well known in the classical theory that $2$-atoms in the
previous atomic decomposition can be replaced by $q$-atoms for any
$1<q\leq\infty$. Let us recall these atoms in the commutative case.
A function $a\in L_1(\Omega)$ is said to be a $q$-atom if there
exist $n\geq1$ and $E\in \mathcal{F}_n$ such that

(i) $\mathbb{E}_na=0$; \ (ii) $\{a\neq0\}\subset E$;\ (iii)
$\|a\|_q\leq\mathbb P(E)^{-1+\frac 1q}.$

\noindent We refer to \cite{Wei94} for more information.

The main difficulty to obtain $q$-atomic decompositions in the
noncommutative case is that the key equivalence (\ref{key}) no
longer holds if one replaces the power indices $2$ by  $q'\neq 2$,
$1\leq q'<\infty$. We overcome this obstacle by Theorem \ref{fcb}.

\begin{theorem}\label{catomicdecomposition}
For all $1<q\leq\infty$,
$$\mathcal{H}_1(\M)=\mathsf{h}^{\mathrm{at}}_{1,q}(\M)$$
with equivalent norms. Here $\mathsf{h}^{\mathrm{at}}_{1,q}(\M)$ is
the $q$-atomic Hardy spaces with its atoms defined as: $a\in
L_1(\M)$ is said to be a $(1,q)$-atom with respect to
$(\M_n)_{n\geq1}$, if there exist $n\geq1$ and a projection
$e\in\mathcal{P}(\M_n)$ such that
\begin{enumerate}[\rm(i)]
\item $\E_n(a)=0$;
\item $r(a)\leq e$ or $l(a)\leq e$;
\item $\|a\|_{q}\leq(\tau(e))^{-\frac{1}{q'}}$.
\end{enumerate}
\end{theorem}

This is exactly the noncommutative analogue of the classical atomic
decomposition. Moreover, applying the conditional version of
John-Nirenberg inequality for $\BMO^c(\M)$ (resp. $\BMO^r(\M)$), we
get a $q$-atomic decomposition for $\mathsf{h}^c_{1}(\M)$ (resp.
$\mathsf{h}^r_{1}(\M)$) with $1<q\leq\infty$ (see Theorem \ref{fine
h1 isometry}), hence recover the $2$-atomic decomposition of
\cite{BCPY} mentioned above.

As in the classical case (see e.g. \cite{CRTW}), we also find some
applications of our results. Indeed, the John-Nirenberg inequality
and atomic decomposition built in this paper have been used in
\cite{HSMP} to establish $H_1\rightarrow L_1$ boundedness of
noncommutative paraproducts or martingale transforms with
noncommuting symbols or coefficients.

Our paper is organized as follows. Section 2 is on preliminaries and
notation. All the results on John-Nirenberg inequality will be
presented in section 3. Section 4 is devoted to the atomic
decomposition of Hardy spaces. In section 5, we answer Junge/Musat's
question in \cite{JuMu07} which implies that the John-Nirenberg
inequality in the classical sense does not hold any more in the
noncommutative setting.

In this article, the letter $c$ always denotes an absolute positive
constant, while $C$ an absolute constant bigger than 1. They may
vary from lines to lines.

\addtocounter{theorem}{-1}
\renewcommand{\thetheorem}{\arabic{section}.\arabic{theorem}}

\section{Preliminaries and notations}

Throughout this paper, we will work on a von Neumann algebra $\M$ with a
normal faithful normalized trace $\tau$. For all $0<p\leq\infty$,
let $L_p(\M,\tau)$ or simply $L_p(\M)$ be the associated
noncommutative $L_p$ spaces. For $x\in L_p(\M)$ we denote the right
and left supports of $x$ by $r(x)$ and $l(x)$ respectively. $r(x)$
(resp. $l(x)$) is also the least projection $e$ such that $xe=x$
(resp. $ex=x$). If $x$ is selfadjoint, $r(x)=l(x)$, denoted by
$s(x)$. We mainly refer the reader to \cite{PiXu03} for more
information on noncommutative $L_p$ spaces.

Let us recall some basic notions on noncommutative martingales. Let
$(\M_n)_{n\geq1}$ be an increasing sequence of von Neumann
subalgebras of $\M$ such that the union of the $\M_n$'s is
$w^*$-dense in $\M$. Let $\mathcal{E}_n$ be the conditional
expectation of $\M$ with respect to $\M_n$. A sequence $x=(x_n)$ in
$L_1(\M)$ is called a noncommutative martingale with respect to
$(\M_n)_{n\geq1}$ if $\mathcal{E}_n(x_{n+1})=x_n$ for every
$n\geq1$. If in addition, all the $x_n$'s are in $L_p(\M)$ for some
$1\leq p\leq\infty$, $x$ is called an $L_p$-martingale. In this case
we set
$$\|x\|_p=\sup_{n\geq1}\|x_n\|_p.$$
If $\|x\|_p<\infty$, $x$ is called a bounded $L_p$-martingale.

Let $x=(x_n)$ be a noncommutative martingale with respect to
$(\M_n)_{n\geq1}$. Define $dx_n=x_n-x_{n-1}$ for $n\geq1$ with the
convention that $x_0=0$ and $\mathcal{E}_0=\mathcal{E}_1$. The
sequence $dx=(dx_n)_n$ is called the martingale difference sequence
of $x$. In the sequel, for any operator $x\in L_1(\M)$ we denote
$x_n=\mathcal{E}_n(x)$ for $n\geq1$.

The sequence $(\M_n)_{n\geq1}$ will be fixed throughout the paper.
All martingales will be with respect to $(\M_n)_{n\geq1}$. Let
$1\leq p<\infty$. Define $\mathcal{H}^c_p$ (resp. $\mathcal{H}^r_p$)
as the completion of all finite $L_p$-martingales under the norm
$\|x\|_{\mathcal{H}^c_p}=\|S_c(x)\|_p$ (resp.
$\|x\|_{\mathcal{H}^r_p}=\|S_r(x)\|_p$), where $S_c(x)$ and $S_r(x)$
are defined as
$$S_c(x)=\big(\sum_{k\geq1}|dx_k|^2\big)^{1/2},\quad S_r(x)=S_c(x^*).$$
The noncommutative martingale Hardy spaces $\mathcal{H}_p(\M)$ are
defined as follows: if $1\leq p<2$,
$$\mathcal{H}_p(\M)=\mathcal{H}^c_p(\M)+\mathcal{H}^r_p(\M)$$
equipped with the norm
$$\|x\|_{\mathcal{H}_p}=\inf_{x=y+z}\{\|y\|_{\mathcal{H}^c_p}+\|z\|_{\mathcal{H}^r_p}\}.$$
When $2\leq p<\infty$,
$$\mathcal{H}_p(\M)=\mathcal{H}^c_p(\M)\cap\mathcal{H}^r_p(\M)$$
equipped with the norm
$$\|x\|_{\mathcal{H}_p}=\max\{\|x\|_{\mathcal{H}^c_p},\|x\|_{\mathcal{H}^r_p}\}.$$

The space $\BMO^c$ is defined as
$$\BMO^c(\M)=\{x\in L_1(\M): \|x\|_{\BMO^c}<\infty\}$$
where
$$\|x\|_{\BMO^c}=\sup_{n\geq1}
\|\mathcal{E}_n|x-x_{n-1}|^2\|^{1/2}_{\infty},$$ and
$$\BMO^r(\M)=\{x:x^*\in\BMO^c(\M)\}.$$
Define
$$\BMO(\M)=\BMO^c(\M)\cap\BMO^r(\M)$$
equipped with the norm
$$\|x\|_{\BMO}=\max\{\|x\|_{\BMO^c},\|x\|_{\BMO^r}\}.$$

Pisier and Xu \cite{PiXu97} proved the two fundamental results:
$\mathcal{H}_p(\M)=L_p(\M)$ and $(\mathcal{H}_1(\M))^*=\BMO(\M)$.
Their work triggered a rapid development of the noncommutative
martingale theory.

We will also work on the conditional version of Hardy
and BMO spaces developed in \cite{JuXu03}. Let $x=(x_n)_{n\geq1}$ be
a finite martingale in $L_2(\M)$. We set
$$s_c(x)=\big(\sum_{k\geq1}\mathcal{E}_{k-1}|dx_k|^2\big)^{1/2}\quad
\mbox{and}\quad s_r(x)=s_c(x^*).$$
Let $0<p<\infty$. Define $\mathsf{h}^c_p(\M)$ (resp. $\mathsf{h}^r_p(\M)$)
as the completion of all finite $L_{\infty}$-martingales under the
(quasi-)norm $\|x\|_{\mathsf{h}^c_p}=\|s_c(x)\|_p$ (resp.
$\|x\|_{\mathsf{h}^r_p}=\|s_r(x)\|_p$). Define $\mathsf{h}^d_p(\M)$
as the subspace of $\ell_p(L_p(\M))$ consisting of all martingale
difference sequences, where $\ell_p(L_p(\M))$ is the space of all
sequences $a=(a_n)_{n\geq1}$ in $L_p(\M)$ such that
$$\|a\|_{\ell_p(L_p(\M))}=\big(\sum_{n\geq1}\|a_n\|^p_p\big)^{1/p}<\infty$$
with the usual modification for $p=\infty$. The noncommutative conditional martingale Hardy spaces are defined
as follows: if $0<p<2$,
$$\mathsf{h}_p(\M)=\mathsf{h}^c_p(\M)+\mathsf{h}^r_p(\M)+\mathsf{h}^d_p(\M)$$
equipped with the (quasi-)norm
$$\|x\|_{\mathsf{h}_p}=\inf_{x=y+z+w}\{\|y\|_{\mathsf{h}^c_p}+\|z\|_{\mathsf{h}^r_p}+\|w\|_{\mathsf{h}^d_p}\}.$$
When $2\leq p<\infty$,
$$\mathsf{h}_p(\M)=\mathsf{h}^c_p(\M)\cap\mathsf{h}^r_p(\M)\cap\mathsf{h}^d_p(\M)$$
equipped with the norm
$$\|x\|_{\mathsf{h}_p}=\max\{\|x\|_{\mathsf{h}^c_p},\|x\|_{\mathsf{h}^r_p},\|x\|_{\mathsf{h}^d_p}\}.$$

The space $\mathsf{bmo}^c$ is defined as
$$\mathsf{bmo}^c(\M)=\{x\in L_1(\M): \|x\|_{\mathsf{bmo}^c}<\infty\}$$
where
$$\|x\|_{\mathsf{bmo}^c}=\max\left\{\|\mathcal{E}_1(x)\|_{\infty},\quad\sup_{n\geq1}
\|\mathcal{E}_n|x-x_{n}|^2\|^{1/2}_{\infty}\right\}.$$ Let
$$\mathsf{bmo}^r(\M)=\{x:x^*\in\mathsf{bmo}^c(\M)\}.$$ Let
$\mathsf{bmo}^d(\M)$ be the subspace of
$\ell_{\infty}(L_{\infty}(\M))$ consisting of all martingale
difference sequences.  Note that $\mathsf{bmo}^d(\M)=\mathsf{h}^d_\infty(\M)$. Define
$$\mathsf{bmo}(\M)=\mathsf{bmo}^c(\M)\cap\mathsf{bmo}^r(\M)\cap\mathsf{bmo}^d(\M)$$
equipped with the norm
$$\|x\|_{\mathsf{bmo}}=\max\{\|x\|_{\mathsf{bmo}^c},\|x\|_{\mathsf{bmo}^r},\|x\|_{\mathsf{bmo}^d}\}.$$
We refer to \cite{JuXu03}, \cite{JuXu08}, \cite{Ran02},
\cite{Ran07}, \cite{JuMe10}, \cite{Per09} for more information on
these spaces.

\addtocounter{theorem}{-1}
\renewcommand{\thetheorem}{\arabic{section}.\arabic{theorem}}

\section{John-Nirenberg inequality}

\subsection{A crude version}

\begin{definition} For $0<p<\infty$, we define
 \begin{enumerate}[\rm(i)]
 \item $$\mathsf{bmo}^c_p(\M)=\big\{x\in
L_1(\M):\|x\|_{\mathsf{bmo}^c_p}<\infty\big\}$$ with
$$\|x\|_{\mathsf{bmo}^c_p}=\max\big\{\|\mathcal{E}_1(x)\|_{\infty},\quad\sup_n\sup_{a\in
\M_n,\|a\|_p\leq1,} \|(x-x_{n})a\|_{\mathsf{h}^c_p}\big\};$$

\item $$\mathsf{bmo}^r_p(\M)=\{x:x^*\in\mathsf{bmo}^c_p(\M)\};$$

\item
$$\mathsf{bmo}_p(\M)=\mathsf{bmo}^c_p(\M)\cap\mathsf{bmo}^r_p(\M)\cap\mathsf{bmo}^d(\M)$$
equipped with the (quasi-)norm
$$\|x\|_{\mathsf{bmo}_p}=\max\{\|x\|_{\mathsf{bmo}^c_p},\|x\|_{\mathsf{bmo}_p^r},\|x\|_{\mathsf{bmo}^d}\}.$$
  \end{enumerate}
\end{definition}

\begin{remark}
When $p=2$, these are exactly the spaces
$\mathsf{bmo}^c(\M)$, $\mathsf{bmo}^r(\M)$ and $\mathsf{bmo}(\M)$.
\end{remark}

Below is our first version of the column (resp. row)  John-Nirenberg inequality.

 \begin{theorem}\label{c-j-n-b}
For all $0<p<\infty$, there exist two constants $\alpha_p$ and
$\beta_p$ such that
$$\alpha^{-1}_p\|x\|_{\mathsf{bmo}^c}\leq\|x\|_{\mathsf{bmo}^c_p}\leq\beta_p\|x\|_{\mathsf{bmo}^c},$$
with $\alpha_p$ and $\beta_p$ satisfying
 \begin{enumerate}[\rm(i)]
 \item  $\alpha_p=1$ for $2\leq p<\infty$;

\item   $\alpha_p\leq C^{1/p-1/2}$ for $0<p<2$;

\item  $\beta_p\leq cp$ for $2\leq p<\infty$;

\item  $\beta_p=1$ for $0<p<2.$
 \end{enumerate}
 The similar inequalities hold for
$\|\cdot\|_{\mathsf{bmo}^r_p}$ and $\|\cdot\|_{\mathsf{bmo}^r}$.
\end{theorem}

\begin{proof} We only need to prove the column case, since the row case can be done by replacing $x$ with $x^*$.
First consider the case $2<p<\infty$. We will show the following
inequalities:
$$\|x\|_{\mathsf{bmo}^c_2}\leq\|x\|_{\mathsf{bmo}^c_p}\leq cp\|x\|_{\mathsf{bmo}^c_2}.$$
The left inequality is obtained directly by H\"{o}lder's inequality.
In fact, taking $a\in \M_n$ with $\|a\|_2\leq1$, there exists a
factorization $a=a_0a_1$ such that $\|a_0\|_p=\|a\|^{2/p}_2\leq1$
and $\|a_1\|_{{2p}/{(p-2)}}=\|a\|^{(p-2)/p}_2\leq1$, so
\begin{align*}
\|(x-x_n)a\|^2_{\mathsf{h}^c_2}&=\tau(a^*_1a^*_0s^2_c(x-x_n)a_0a_1)\\
&\leq\|a^*_1\|_{\frac{2p}{p-2}}\|a_0^*s^2_c(x-x_n)a_0\|_{\frac{p}{2}}\|a_1\|_{\frac{2p}{p-2}}\\
&\leq\|(x-x_n)a_0\|^2_{\mathsf h^c_p}.
\end{align*}
We invoke complex interpolation to prove the right inequality. Fix
$n$, let $b\in{L_p(\M_n)}$ with $\|b\|_p\leq1$ and
$S=\{z\in\mathbb{C}: 0\leq Rez\leq1\}$. Then by interpolation
between $L_p$ spaces $L_p=(L_2,L_{\infty})_{\theta}$, there exists
an operator-valued function $B$ which is continuous on $S$ and
analytic in the interior of $S$ such that $B(\theta)=b$ and
$$\sup_{t\in\mathbb{R}}\|B(it)\|_{2}\leq1,\qquad
\sup_{t\in\mathbb{R}}\|B(1+it)\|_{{\infty}}\leq1.$$ Define
$$f(z)=(x-x_n)B(z).$$ Then on the one hand, by the definition of $\mathsf{bmo}^c_2(\M)$, we have
$$\|f(it)\|_{\mathsf{h}^c_2}\leq\|x\|_{\mathsf{bmo}^c_2}.$$ On the other hand,
by a simple calculation, we have
$$\|f(1+it)\|_{\mathsf{bmo}^c_2}\leq\|x-x_n\|_{\mathsf{bmo}^c_2}\|B(1+it)\|_{\infty}\leq \|x\|_{\mathsf{bmo}^c_2}.$$
Therefore, by interpolation,
$$\|f(\theta)\|_{(\mathsf{h}^c_2,\mathsf{bmo}^c)_{\theta}}\leq\|x\|_{\sf{bmo}^c_2}=\|x\|_{\sf{bmo}^c}.$$
However by \cite{BCPY},
$$(\mathsf{h}^c_2,\mathsf{bmo}^c)_{\theta}\subset\mathsf{h}^c_p$$
with relevant constant majorized by $cp$. We then deduce that
\begin{align}
\|f(\theta)\|_{\mathsf{h}^c_p}\leq cp\|x\|_{\mathsf{bmo}^c},
\label{c-c-b-const}
\end{align}
hence the desired inequality holds.

For the case $0<p<2$. We show the following inequalities:
$$\|x\|_{\mathsf{bmo}^c_p}\leq\|x\|_{\mathsf{bmo}^c_2}\leq C^{1/p-1/2}\|x\|_{\mathsf{bmo}^c_p}.$$
Again, the left inequality is obtained by H\"{o}lder's inequality.
It remains to prove the right one. We choose $2<p_1<\infty$ and
$0<\theta<1$ such that ${1}/{2}={(1-\theta)}/{p}+{\theta}/{p_1}$.
Fix $n$, by the definition of $\mathsf{bmo}^c_{p}(\M)$, we can view
$x-x_n$ as a bounded operator from $L_p(\M_n)$ to
$\mathsf{h}^c_p(\M)$. Then we have the following two inequalities:
$$\|x-x_n\|_{L_p(\M_n)\rightarrow \mathsf{h}^c_p}\leq\|x\|_{\mathsf{bmo}^c_p},\qquad\|x-x_n\|_{L_{p_1}(\M_n)\rightarrow \mathsf{h}^c_{p_1}}\leq\|x\|_{\mathsf{bmo}^c_{p_1}}.$$
Then by interpolation, we get
$$\|x-x_n\|_{L_{2}(\M_n)\rightarrow (\mathsf{h}^c_p,\mathsf{h}^c_{p_1})_{\theta}}\leq \|x\|^{1-\theta}_{\mathsf{bmo}^c_p}\|x\|^{\theta}_{\mathsf{bmo}^c_{p_1}}.$$
Now by the trivial contractive inclusion
$(\mathsf{h}^c_p,\mathsf{h}^c_{p_1})_{\theta}\subset
\mathsf{h}^c_2$, and the right inequality in the case
$2<p_1<\infty$, we get
$$\|x-x_n\|_{L_{2}(\M_n)\rightarrow \mathsf{h}^c_2}\leq cp_1\|x\|^{1-\theta}_{\mathsf{bmo}^c_p}\|x\|^{\theta}_{\mathsf{bmo}^c_{2}}.$$
Therefore,
$$\|x\|_{\mathsf{bmo}^c_{2}}\leq (cp_1)^{\theta}\|x\|^{1-\theta}_{\mathsf{bmo}^c_p}\|x\|^{\theta}_{\mathsf{bmo}^c_{2}},$$
hence
$$\|x\|_{\mathsf{bmo}^c_{2}}\leq (cp_1)^{\frac{\theta}{1-\theta}}\|x\|_{\mathsf{bmo}^c_p}.$$
Noting that ${\theta}/({1-\theta})=({1/p-1/2})/({1/2-1/p_1})$, we
get the desired estimate by taking $C=(cp_1)^{1/(1/2-1/p_1)}$.
\end{proof}

\begin{remark} The constant in (\ref{c-c-b-const}) is optimal. This can be seen as follows.
By Lemma 4.3 in \cite{BCPY}, $\mathsf h^c_{p'}(\M)$ embeds into
$(\mathsf h^c_2(\M),\mathsf h^c_1(\M))_{\theta}$ with constant
independent of $p'$. So $((\mathsf h^c_2(\M))^*,(\mathsf
h^c_1(\M))^*)_{\theta}$ embeds into $(\mathsf h^c_{p'}(\M))^*$ with
constant independent of $p$ by duality. Finally, by the optimal
embedding $(\mathsf h^c_{p'}(\M))^*\subset \mathsf h^c_p(\M)$ with
constant $cp$ in \cite{JuXu03} and
$\mathsf{bmo}^c(\M)\subset(\mathsf h^c_1(\M))^*$ in \cite{Per09},
$(\mathsf h^c_2(\M),\mathsf{bmo}^c(\M))_{\theta}$ embeds into
$\mathsf h^c_{p}(\M)$ with optimal constant $cp$.
\end{remark}
It is natural to ask whether there is a result similar to Theorem
\ref{c-j-n-b} for $\BMO^c$ by replacing $\mathsf{h}^c_p$ and
$x-x_{n}$ in the definition of $\mathsf{bmo}^c_p$ by
$\mathcal{H}^c_p$ and $x-x_{n-1}$ respectively. Using the identity
$$\BMO^c(\M)\simeq\mathsf{bmo}^c(\M)\cap\mathsf{bmo}^d(\M)$$
proved in \cite{Per09}, we are reduced to deal with the diagonal
space $\mathsf{bmo}^d(\M)$. Surprisingly, the result is true only
for $2\leq p<\infty$ (see Remark \ref{counter example1}).

\begin{definition}

For $1\leq p<\infty$, we define
\begin{enumerate}[(i)]
\item $$\BMO^c_p(\M)=\left\{x\in
L_1(\M):\|x\|_{\BMO^c_p}<\infty\right\}$$ with
$$\|x\|_{\BMO^c_p}=\sup_n\sup_{a\in
\M_n,\|a\|_p\leq1} \|(x-x_{n-1})a\|_{\mathcal{H}^c_p};$$

\item $$\BMO^r_p(\M)=\{x:x^*\in\BMO^c_p(\M)\};$$

\item
$$\BMO_p(\M)=\BMO^c_p(\M)\cap\BMO^r_p(\M)$$
equipped with the norm
$$\|x\|_{\BMO_p}=\max\{\|x\|_{\BMO^c_p},\|x\|_{\BMO_p^r}\}.$$
\end{enumerate}

\end{definition}

\begin{remark}
For $p=2$, we recover the spaces $\BMO^c(\M)$, $\BMO^r(\M)$ and
$\BMO(\M)$.
\end{remark}

The following lemma will alow us to handle with the diagonal space
$\mathsf{bmo}^d(\M)$.

\begin{lemma}\label{lemma}
For $2\leq p<\infty ,$ we have
\[
cp^{-1}\|b\|_{{\infty}}\leq \sup_{a\in\M,\|a\|_p\leq
1}\|ba\|_{{\mathcal H}^c_p}\leq cp^{\frac 12}\|b\|_{{\infty}}.
\]
\end{lemma}

\begin{proof} Note that $\|\cdot \|_{{\mathcal H}^c_p}\leq cp^{1/2}
\|\cdot \|_{p}$ (see \cite{Ran02}, Remark 5.4 as a reference for
the constant we use here), we have
\[
\sup_{a\in\M,\|a\|_p\leq 1}\|ba\|_{{\mathcal H}^c_p}\leq cp^{\frac
12}\sup_{a\in\M,\|a\|_p\leq 1}\|ba\|_{p}=cp^{\frac
12}\|b\|_{\infty}.
\]
For the first inequality, without loss of generality assume
$\|b\|_{\infty}=1.$ Note that for selfadjoint $x\in $
${\M},\|x\|_{p}\leq cp\|x\|_{{\mathcal H}_p^c}$ (see \cite{Ran02},
Remark 5.4). Then
\begin{align*}
\|b^{*}\|_{{\infty}} &=\sup_{y\in \M,\|y\|_{2p}\leq 1}\|yb^{*}\|_{{2p}} \\
&=\sup_{y\in\M,\|y\|_{2p}\leq 1}\|b|y|^2b^{*}\|_{p}^{\frac 12} \\
&\leq cp^{\frac 12}\sup_{y\in\M,\|y\|_{2p}\leq 1}\|b|y|^2b^{*}\|_{{\mathcal H}^c_p}^{\frac 12} \\
&\leq cp^{\frac 12}\sup_{a\in\M, \|a\|_p\leq 1}\|ba\|_{{\mathcal
H}_p^c}^{\frac 12}.
\end{align*}
And then $cp^{-1}\|b\|_{{\infty}}\leq \sup_{a\in\M,\|a\|_p\leq 1}\|ba\|_{{\mathcal H}%
^c_p}.$
\end{proof}

\begin{theorem}\label{c-j-n-B}
For all $2\leq p<\infty$, we have
$$\BMO^c_p(\M)=\BMO^c(\M)$$
with equivalent norms. More precisely,
$$cp^{-1}\|x\|_{{\BMO}^{c}}\leq\|x\|_{{\BMO}
_p^{c}}\leq cp\|x\|_{{\BMO}^{c}}. $$ Similarly,
$\BMO^r_p(\M)=\BMO^r(\M)$ with equivalent norms.
\end{theorem}
Using the previous lemma and the identity
$\BMO^c(\M)\simeq\mathsf{bmo}^c(\M)\cap\mathsf{bmo}^d(\M)$, we can
easily deduce Theorem \ref{c-j-n-B} from Theorem \ref{c-j-n-b}. We
will however present a direct proof.
\begin{proof} We only prove the inequalities for the column case, the row
case can be dealt with similarly. By the previous lemma and
H\"{o}lder's inequality, we have
\begin{align*}
\|\mathcal{E}_n\sum_{k=n}^\infty
|dx_k|^2\|_{{\infty}}
&\leq \sup_{b\in\M^+_n,\|b\|_1\leq 1}\tau \left(\sum_{k=n+1}^\infty
|dx_k|^2b\right)+\|x_n-x_{n-1}\|^2_{{\infty}} \\
&\leq\sup_{b\in\M^+_n,\|b\|_1\leq 1}\tau \left(\sum_{k=n+1}^\infty
|(dx_k)b^{\frac 1p}|^2b^{
\frac{p-2}p}\right)\\
&\quad\quad\quad\quad\quad\quad+cp^2\sup_{a\in\M_n,\|a\|_p\leq 1}\|(x_n-x_{n-1})a\|^2_{{\mathcal H}^c_p} \\
&\leq\sup_{b\in\M^+_n,\|b\|_1\leq 1}\left\| \sum_{k=n+1}^\infty
|(dx_k)b^{\frac 1p}|^2\right\| _{{\frac p2}}\left\|
b^{\frac{p-2}p}\right\| _{{(\frac
p2)^{\prime }}}\\
&\quad\quad\quad\quad\quad\quad+cp^2\sup_{a\in\M_n,\|a\|_p\leq 1}\|(x_n-x_{n-1})a\|^2_{{\mathcal H}_p^c} \\
&\leq\sup_{b\in\M^+_n,\|b\|_1\leq 1}\left\| (x-x_n)b^{\frac
1p}\right\|^2 _{{\mathcal H}^c_p}\\
&\quad\quad\quad\quad\quad\quad+cp^2\sup_{a\in\M_n,\|a\|_p\leq
1}\|(x_n-x_{n-1})a\|^2_{{\mathcal H}^c_p}.
\end{align*}
Then by
$\|\mathcal{E}_nx\|_{\mathcal{H}^c_p}\leq\|x\|_{\mathcal{H}^c_p}$,
$$\|x\|_{{\BMO}_2^{c}}\leq cp\sup_{a\in\M_n,\|a\|_p\leq 1}\left\| (x-x_{n-1})a\right\|
_{{\mathcal H}^c_p}=cp\|x\|_{{\BMO}^c_p}.$$ Conversely, by the
previous lemma,
\begin{align}
\|x\|_{{\BMO}_p^{c}} &\leq\sup_n\sup_{a\in\M_n,\|a\|_p\leq1}\left\|
(x-x_n)a\right\| _{{\mathcal H}^c_p}\nonumber\\
&\quad\quad\quad\quad\quad\quad+\sup_n\sup_{a\in\M_n,\|a\|_p\leq1}\|(x_n-x_{n-1})a\|_{{\mathcal
H}^c_p}
\nonumber \\
&\leq\sup_n\sup_{a\in\M_n,\|a\|_p\leq1}\left\| (x-x_n)a\right\|
_{{\mathcal H}^c_p}+cp^{\frac 12}\sup_n\|x_n-x_{n-1}\|_{{\infty}}
\nonumber\\
&\leq \sup_n\sup_{a\in\M_n,\|a\|_p\leq1}\left\| (dx_ka)_{k=n+1}^\infty
\right\|_{L_p(\ell^c_2)}+cp^{\frac 12}\|x\|_{{\BMO}^c_2}.\label{a}
\end{align}
Note that, by the Hahn-Banach theorem and the duality between
${\mathcal H}^c_1({\M})$ and $\BMO^{c}({\M})$,  there exists a
sequence $(b_n)_{n=1}^\infty \in L_\infty ({\M};\ell_2^c)$ such that
\[
\left\| (b_n)_{n=1}^\infty \right\|
_{L_\infty(\ell^c_2)}=\|x\|_{{\BMO}^c},\quad dx_k=\mathcal
E_kb_k-\mathcal E_{k-1}b_k.
\]

Thus by the noncommutative Stein inequality (see \cite{Ran02} for
the constant used below) and H\"{o}lder's inequality,
\begin{align*}
&\sup_{a\in\M_n,\|a\|_p\leq1}\left\| (dx_ka)_{k=n+1}^\infty \right\|
_{L_p(\ell^c_2)}\\
&\quad\quad\quad\leq \sup_{a\in\M_n,\|a\|_p\leq1}\left\| (\mathcal
E_k(b_ka))_{k=n+1}^\infty \right\| _{L_p(\ell
^c_2)}\\
&\quad\quad\quad\quad\quad\quad\quad+\sup_{a\in\M_n,\|a\|_p\leq1}\left\|
(\mathcal E_kb_{k+1}a)_{k=n}^\infty \right\|
_{L_p(\ell^c_2)} \\
&\quad\quad\quad\leq cp\sup_{a\in\M_n,\|a\|_p\leq1}\left\|
(b_ka)_{k=n+1}^\infty \right\| _{L_p(\ell^c_2)}\\
&\quad\quad\quad\leq cp\left\| \sum_{k=1}^\infty |b_k|^2\right\|
_{\infty}^{\frac 12}=cp\|x\|_{{\BMO}_2^{c}}.
\end{align*}
Combining this with (\ref{a}) we finish the proof.
\end{proof}

\begin{remark}\label{counter example1}
It is a bit surprising that Theorem \ref{c-j-n-B} is actually wrong
for any $p<2$. Indeed, choose a filtration $\M_1$, $\M_2$,
$\M_3$,...,$\M_{n-1}$ and $y\in \M_{n-1}$ such that $\|y\|_p=1$ and
$\|y\|_{\mathcal{H}_p^c}=c_n>>1$. Let $\M_{n}=L_{\infty}(\Omega,
\M_{n-1})$ with $\Omega=\{0,1\}$ with $\mu\{1\}=\mu\{0\}=1/2$. We
certainly can view $\M_k$, $k<n$ as the space of constant functions
on $\Omega$, so $\M_k\subset \M_n$. Let $x=1$ on $\{0\}$ and $x=-1$
on $\{1\}$ then $x_{n-1}=0$. Let $a=y$ on $\{0\}$ and $a=-y$ on
$\{1\}$. Then $(x-x_{n-1})a=y$ whose $\mathcal{H}_p^c$ norm equals
$c_n$ and $\|a\|_{p}=1$, so $\|x\|_{\BMO^c_p}\geq c_n$. But
$\|x\|_{\BMO^c_2}=1$.
\end{remark}

In the rest of this subsection, we turn to Junge/Musat's type of
John-Nirenberg inequality. In \cite{JuMu07}, Junge and Musat
established the inequality for $2<p<\infty$ in the state case. Later
the second author of the present paper gave a simple proof for all
$1\leq p<\infty$ in the tracial setting (see \cite{Mei06}). The idea
of the proof of Theorem \ref{c-j-n-b} can be applied to obtain this
inequality for all $0<p<\infty$ (see Corollary \ref{c-jm-j-n-B}). We
start again with $\mathsf{bmo}(\M)$.

\begin{theorem}\label{c-jm-j-n-b}
For all $0<p<\infty$, we have
$$\alpha^{-1}_p\|x\|_{\mathsf{bmo}}\leq\mathsf b_p(x)\leq \beta_p\|x\|_{\mathsf{bmo}}$$
where
\begin{align*}
\mathsf b_p(x)=\max\{&\sup_{n}\|(dx_n)_n\|_{\infty},\
\sup_n\sup_{b\in\M_n,\|b\|_p\leq1} \|(x-x_{n})b\|_{p},\\
&\sup_n\sup_{b\in\M_n,\|b\|_p\leq1}
\|b(x-x_{n})\|_{p}\}.
\end{align*}
The constant $\alpha_p$ and $\beta_p$ have the same orders as those
in Theorem \ref{c-j-n-b}.
\end{theorem}

\begin{proof}
We first treat the case $2\leq p<\infty$. For $p=2$, it is trivial.
So we can assume $2<p<\infty$. The inequality
$$\|x\|_{\mathsf{bmo}}\leq\mathsf b_p(x)$$
follows from H\"{o}lder's inequality. We will prove the reverse
inequality by interpolation. By a simple calculation, we have the
following estimates
$$\|(x-x_n)b\|_{\mathsf{bmo}^c}\leq\|x\|_{\mathsf{bmo}^c}\|b\|_{{\infty}},$$
$$\|(x-x_n)b\|_{\mathsf{bmo}^r}\leq\|x\|_{\mathsf{bmo}^r}\|b\|_{{\infty}},$$
$$\|(x-x_n)b\|_{\mathsf{bmo}^d}\leq\|x\|_{\mathsf{bmo}^d}\|b\|_{{\infty}}.$$
Then it follows that
$$\|(x-x_n)b\|_{\mathsf{bmo}}\leq\|x\|_{\mathsf{bmo}}\|b\|_{{\infty}}.$$
On the other hand, it is clear that
$$\|(x-x_n)b\|_{2}=\|(x-x_n)b\|_{\mathsf{h}^c_2}\leq\|x\|_{\mathsf{bmo}}\|b\|_{2}.$$
Then by the interpolation result of \cite{BCPY}, we have
\begin{align}
\|(x-x_n)b\|_{p}&\leq cp\|(x-x_n)b\|_{(L_2,\mathsf{bmo})_{\theta}} \label{c-b-const}\\
&\leq cp\|x\|_{\mathsf{bmo}}\|b\|_{p}.\nonumber
\end{align}
In the same way, we obtain
$$\|b(x-x_n)\|_{p}\leq cp\|x\|_{\mathsf{bmo}}\|b\|_{p}.$$
Thus we prove the assertion.

Now we turn to the case $0<p<2$, by H\"older's inequality, we obtain
the trivial part
$$\mathsf b_p(x)\leq \mathsf b_2(x)=\|x\|_{\mathsf{bmo}}.$$
Let us prove the inverse one,  let $2<p_1<\infty$ and $\theta$ be
such that
$$\frac{1}{2}=\frac{1-\theta}{p}+\frac{\theta}{p_1}.$$ We view $x-x_n$
and $(x-x_n)^*$ as two operators. By interpolation,
\begin{align*}
&\|(x-x_n)\|_{L_2(\M_n) \rightarrow
L_2(\M)}\\&\leq\|(x-x_n)\|^{1-\theta}_{L_p(\M_n)\rightarrow
L_p(\M)}\|(x-x_n)\|^{\theta}_{L_{p_1}(\M_n)\rightarrow
L_{p_1}(\M)}\\
\end{align*}
 and similarly for $(x-x_n)^*$.
By the estimate for $p_1>2$, we have
$$\mathsf b_2(x)\leq (cp_1)^{\theta}\mathsf b^{1-\theta}_p(x)\mathsf b^{\theta}_2(x).$$
Therefore, we obtain
$$\|x\|_{\mathsf{bmo}}\leq(cp_1)^{\frac{\theta}{1-\theta}}\mathsf b_p(x)=C^{{1/p}-{1/2}}\mathsf b_p(x),$$
with $C=(cp_1)^{1/(1/2-1/p_1)}$.
\end{proof}

\begin{remark}
The constant in \eqref{c-b-const} is optimal. This can be seen as
follows. By Lemma 4.3 in \cite{BCPY}, $\mathsf h^c_{p'}(\M)$ embeds
into $(\mathsf h^c_2(\M),\mathsf h^c_1(\M))_{\theta}$ with constant
independent of $p'$. So $\mathsf h_{p'}(\M)$ embeds into $(\mathsf
h_2(\M),\mathsf h_1(\M))_{\theta}$ with constant independent of
$p'$. Now by Theorem 4.1 in \cite{Ran07}, $L_{p'}(\M)$ embeds into
$\mathsf h_{p'}(\M)$, hence into $(\mathsf h_2(\M),\sf
h_1(\M))_{\theta}$ with optimal constant $c/(p'-1)$. Then by
duality, $((\mathsf h_2(\M))^*,(\mathsf h_1(\M))^*)_{\theta}$ embeds
into $(L_{p'}(\M))^*=L_p(\M)$ with best constant $cp$. At last, by
$\mathsf{bmo}(\M)\subset(\mathsf h_1(\M))^*$ in \cite{Per09},
$(\mathsf h_2(\M),\mathsf{bmo}(\M))_{\theta}$ embeds into
$L_{p}(\M)$ with optimal constant $cp$.
\end{remark}

\begin{remark}\label{cjnb-t-c-jmb}
We can directly compare the norms $\|\cdot\|_{\mathsf{bmo}_p}$ and
$\mathsf b_p(\cdot)$ directly for $1<p<\infty$ by using Theorem
\ref{c-j-n-b}.
 \end{remark}

Let us justify this remark. We first deal with the case
$2<p<\infty$. Fix $n$, for any $b\in \M_n$ with $\|b\|_{p}\leq1$, by
the noncommutative Burkholder inequality \cite{JuXu03}, we have
$$\|(x-x_n)b\|_{\mathsf{h}^c_p}\leq cp\|(x-x_n)b\|_{p},\quad\|b(x-x_n)\|_{\mathsf{h}^r_p}\leq cp\|b(x-x_n)\|_{p},$$
hence
$$\|(x-x_n)b\|_{\mathsf{h}^c_p},\;\|b(x-x_n)\|_{\mathsf{h}^r_p}\leq
cp\mathsf b_p(x)$$ Then by Theorem \ref{c-j-n-b},
\begin{align*}
\|x\|_{\mathsf{bmo}_p}\leq cp\mathsf b_p(x).
\end{align*}
Another direction can be done by the way in Theorem
\ref{c-jm-j-n-b},
$$\mathsf{b}_p(x)\leq cp\|x\|_{\sf{bmo}}\leq cp\|x\|_{\mathsf{bmo}_p}.$$
 For the case $1<p<2$. The trivial part
\begin{align*}
\mathsf b_p(x)\leq c\|x\|_{\mathsf{bmo}_p}
\end{align*}
follows from the noncommutative Burkholder inequality in
\cite{JuXu03}. Now let us prove the inverse one. Take $b\in \M_n$
with $\|b\|_2\leq1$. By H\"older's inequality, we have
\begin{align*}
&\|(x-x_n)b\|^2_{2}=\tau(b^{2/{p'}}(x-x_n)^*(x-x_n)b^{2/p})\\
&\leq\|b^{2/{p'}}(x-x_n)^*\|_{{p'}}\|(x-x_n)b^{2/p}\|_{p}
\end{align*}
and
\begin{align*}
&\|b(x-x_n)\|^2_{2}=\tau((x-x_n)^*b^{2/{p'}}b^{2/p}(x-x_n))\\
&\leq\|(x-x_n)^*b^{2/{p'}}\|_{{p'}}\|b^{2/p}(x-x_n)\|_{p}.
\end{align*}
So by the result in Theorem \ref{c-j-n-b} for $2<p'<\infty$, we have
\begin{align*}
&\|b(x-x_n)\|^2_{2},\|(x-x_n)b\|^2_{2}\\
&\leq\max\{\|b^{2/{p'}}(x-x_n)^*\|_{{p'}},\|(x-x_n)^*b^{2/{p'}}\|_{{p'}}\}\\
&\quad\cdot\max\{\|(x-x_n)b^{2/p}\|_{p},\|b^{2/p}(x-x_n)\|_{p}\}\\
&\leq c\|x\|_{\mathsf{bmo}_{p'}}\cdot \mathsf b_p(x)\leq
cp'\|x\|_{\mathsf{bmo}_{2}}\cdot \mathsf b_p(x)
\end{align*}
Then by the definition of $\mathsf{bmo}_2(\M)$, we finish the proof
by Theorem \ref{c-j-n-b}
\begin{align*}
&\|x\|_{\mathsf{bmo}_p}\leq \|x\|_{\mathsf{bmo}_2}\leq cp'\mathsf
b_p(x).
\end{align*}

The following corollary extends Junge/Musat's theorem to all $0<p<\infty$. It can be proved similarly as  Theorem \ref{c-j-n-b}. However, using the identity $\BMO(\M)\simeq\mathsf{bmo}(\M)$
proved in \cite{Per09}, we give a simpler proof.

\begin{corollary}\label{c-jm-j-n-B}
For $0<p<\infty$, we have
$$\alpha^{-1}_p\|x\|_{\BMO}\leq\mathcal{B}_p(x)\leq \beta_p\|x\|_{\BMO},$$
where
\begin{align*}
\mathcal{B}_p(x)=\max\{
&\sup_n\sup_{b\in\M_n,\|b\|_p\leq1} \|(x-x_{n-1})b\|_{p},\\
&\sup_n\sup_{b\in\M_n,\|b\|_p\leq1}
\|b(x-x_{n-1})\|_{p}\}.
\end{align*}
The constant $\alpha_p$ and $\beta_p$ have the same orders as those
in Theorem \ref{c-j-n-b}.
\end{corollary}

\begin{proof}
For $2\leq p<\infty$, it is very easy to get
$$\mathcal{B}_p(x)\leq \mathrm b_p(x)\leq cp\|x\|_{\mathsf{bmo}}\leq cp\|x\|_{\mathcal{BMO}}$$
from the triangular inequality
$$\|(x-x_{n-1})b\|_p\leq\|(x-x_n)b\|_p+\|(x_n-x_{n-1})b\|_p,$$
with $b\in \M_n$ and $\|b\|_p\leq1$. And the rest of the proof is
the same to Theorem \ref{c-jm-j-n-b}.
\end{proof}

\begin{remark}\label{counter example2}
The following example 
shows that Junge/Musat's John-Nirenberg inequality does not hold for
$\mathsf{bmo}^c$ or $\BMO^c$. The example is the same as the one
given in Remark 3.20 of \cite{JuMu07}. Let $n$ be a positive integer
and consider the von Neumann algebra
$$\M=L_{\infty}(\mathbb{T})\bar{\otimes}M_{n},$$
where $M_n$ is the algebra of $n\times n$ matrices with normalized
trace. For $k\geq1$ let ${\mathcal F}_k$ be the $\sigma$-algebra
generated by dyadic intervals in $\mathbb{T}$ of length $2^{-k}$.
Denote by $\M_k$ the subalgebra $L_{\infty}(\mathbb{T},{\mathcal
F}_k)\bar{\otimes}M_{n}$ of $\M$ and let $\mathcal{E}_k=\mathbb
E_k\otimes id_{M_n}$ be the conditional expectation onto $\M_k$. Let
$r_k$ be the $k$-th Rademacher function on $\mathbb{T}$ and consider
$$x=\sum^n_{k=1}r_k\otimes e_{1k}.$$
Then $x$ is a martingale relative to the filtration
$(\mathcal{M}_k)_{k\geq1}$ and the martingale differences are given
by $dx_k=r_k\otimes e_{1k}$. A simple calculation shows that
$$\sup_m\|x-x_m\|_p=(n-1)^{\frac 12}n^{-\frac 1p},$$
while
$$\|x\|_{\mathsf{bmo}^c}=\sup_m\left\|\sum^n_{k=m+1}\mathcal{E}_m|d_kx|^2\right\|^{\frac 12}_{\infty}=1.$$
Let $p>2$. Then for any $c>0$, there exists $n\geq1$ such that
$(n-1)^{1/2}n^{-1/p}>c$. Hence
$$\sup_m\sup_{b\in \M_m,\|b\|p\leq1}\|(x-x_m)b\|_p\geq\sup_m\|x-x_m\|_p>>\|x\|_{\mathsf{bmo}^c}.$$
\end{remark}

\subsection{A fine version}

Now we can formulate the fine version of the column (resp. row)
John-Nirenberg inequality.

\begin{definition} For $0<p<\infty$, we define
 $$\mathsf{bmo}^c_{p,\rm{pr}}(\M)=\big\{x\in
 L_1(\M):\|x\|_{\mathsf{bmo}^c_{p,\rm{pr}}}<\infty\big\}$$
with
 $$\|x\|_{\mathsf{bmo}^c_{p,\rm{pr}}}=\max\big\{\|\mathcal{E}_1(x)\|_{\infty},\quad\sup_n\sup_{e\in \mathcal{P}(\M_n)}
\|(x-x_{n})\frac{e}{(\tau(e))^{{1}/{p}}}\|_{\mathsf{h}^c_{p}}\big\}.$$
Similarly,
$$\mathsf{bmo}^r_{p,\rm{pr}}(\M)=\{x:x^*\in\mathsf{bmo}^c_{p,\rm{pr}}(\M)\}\;\textrm{ with }\;
\|x\|_{\mathsf{bmo}_{p,\rm{pr}}^r}=\|x^*\|_{\mathsf{bmo}_{p,\rm{pr}}^c}.$$
Finally,
$$\mathsf{bmo}_{p,\rm{pr}}(\M)=\mathsf{bmo}^c_{p,\rm{pr}}(\M)\cap\mathsf{bmo}^r_{p,\rm{pr}}(\M)\cap\mathsf{bmo}^d(\M)$$
equipped with
$$\|x\|_{\mathsf{bmo}_{p,\rm{pr}}}=\max\{\|x\|_{\mathsf{bmo}^c_{p,\rm{pr}}},\|x\|_{\mathsf{bmo}_{p,\rm{pr}}^r},\|x\|_{\mathsf{bmo}^d}\}.$$

\end{definition}

The fine version of the column (resp. row) John-Nirenberg inequality
is stated as follows.
\begin{theorem}\label{f-j-n-b} For all $0<p<\infty$, we have
$$\alpha^{-1}_p\|x\|_{\mathsf{bmo}^c}\leq\|x\|_{\mathsf{bmo}^c_{p,\rm{pr}}}\leq\beta_p\|x\|_{\mathsf{bmo}^c}.$$
The constants $\alpha_p$ and $\beta_p$ have the same properties as
those in Theorem \ref{c-j-n-b}. The same inequalities hold for
$\|\cdot\|_{\mathsf{bmo}^r}$ and
$\|\cdot\|_{\mathsf{bmo}^r_{p,\rm{pr}}}$.
\end{theorem}

\begin{proof}
We first consider the case $0<p\leq1$. By Theorem \ref{c-j-n-b}, the
trivial part
$$\|x\|_{\mathsf{bmo}^c_{p,\mathrm{pr}}}\leq\|x\|_{\mathsf{bmo}^c_{p}}\leq\|x\|_{\mathrm{bmo}^c}$$
follows from the fact that ${e}/{(\tau(e))^{1/p}}\in \M_n$ and its
$L_p$-norm equals 1. Now we turn to the proof of the inverse
inequality. Since any $a\in \M_n$ with $\|a\|_p\leq1$ can be
approximated by sums $\sum_k \lambda_k{e_k}/{(\tau(e_k))^{1/p}}$
with $e_k$'s in $\M_n$ and $\sum_k|\lambda_k|^p\leq1$. Thus we can
assume that $a$ itself is such a sum. Then
\begin{align*}
\|(x-x_n)a\|^p_{\mathsf
h^c_p}&=\|\sum_k\lambda_k(x-x_n)\frac{e_k}{(\tau(e_k))^{1/p}}\|^p_{\mathsf
h^c_p}\\
&\leq
\sum_k|\lambda_k|^p\|(x-x_n)\frac{e_k}{(\tau(e_k))^{1/p}}\|^p_{\mathsf
h^c_p}\\
&\leq
\sum_k|\lambda_k|^p\|x\|^p_{\mathsf{bmo}^c_{p,\mathrm{pr}}}\leq\|x\|^p_{\mathsf{bmo}^c_{p,\mathrm{pr}}}.
\end{align*}
Therefore by Theorem \ref{c-j-n-b},
$$\|x\|_{\mathsf{bmo}^c}\leq C^{1/p-1/2}\|x\|_{\mathsf{bmo}^c_p}\leq C^{1/p-1/2}\|x\|_{\mathsf{bmo}^c_{p,\rm{pr}}}.$$

Now let $1<p<\infty$. Again, because of the fact that
${e}/{(\tau(e))^{1/p}}\in \M_n$ and its $L_p$-norm equals 1, by
Theorem \ref{c-j-n-b},
\begin{align}\label{c-best}
\|x\|_{\mathsf{bmo}^c_{p,\mathrm{pr}}}\leq\|x\|_{\mathsf{bmo}^c_{p}}\leq
c_1p\|x\|_{\mathsf{bmo}^c}.
\end{align}
We exploit the result for $p=1$ to prove the inverse inequality. By
H\"{o}lder's inequality, we have
\begin{align*}
\|x\|_{\mathsf{bmo}^c_{1,\mathrm{pr}}}\leq\|x\|_{\mathsf{bmo}^c_{p,\mathrm{pr}}}.
\end{align*}
We end the proof by Theorem \ref{c-j-n-b} and the result for $p=1$,
\begin{align*}
\|x\|_{\mathsf{bmo}^c}\leq C^{1/p-1/2}\|x\|_{\mathsf{bmo}^c_{1}}\leq
C^{1/p-1/2}\|x\|_{\mathsf{bmo}^c_{1,\mathrm{pr}}}\leq
C^{1/p-1/2}\|x\|_{\mathsf{bmo}^c_{p,\mathrm{pr}}}.
\end{align*}
\end{proof}

Now we give the distributional form of the John-Nirenberg inequality
for $\mathsf{bmo}^c(\M)$ and $\mathsf{bmo}^r(\M)$.

\begin{theorem}\label{c-weak}
Let $x\in \mathsf{bmo}^c(\M)$. Then for all natural numbers $n\geq1$,
all $e\in \mathcal{P}(\M_n)$ and for all $\lambda>0$, we have
$$\frac{1}{\tau(e)}\tau(\mathds{1}_{(\lambda,\infty)}(s_c((x-x_n)e)))\leq 2\exp(-\frac{c\lambda}{\|x\|_{\mathsf{bmo}^c}}),$$
with $c$ an absolute constant. Here
$\mathds{1}_{(\lambda,\infty)}(a)$ denotes the spectral projection
of a positive operator $a$ corresponding to the interval
$(\lambda,\infty)$.
\end{theorem}

\begin{proof}
By homogeneity, we can assume $\|x\|_{\mathsf{bmo}^c}=1$. We first
deal with the case $\lambda\geq 2c_1$, where $c_1$ is the constant
in inequality (\ref{c-best}). Let $p=\lambda/{(2c_1)}\geq1$, by
Chebychev's inequality and Theorem \ref{f-j-n-b},
\begin{align*}
&\tau(\mathds{1}_{(\lambda,\infty)}(s_c((x-x_n)e)))\leq
\tau(e)\frac{\|(x-x_n)e\|^p_{\mathsf h^c_p}}{\lambda^p}\\
&\leq
\tau(e)(c_1p\lambda^{-1})^p=\tau(e)\exp(p\ln(c_1p\lambda^{-1}))=\tau(e)\exp(-\frac{\ln2}{2c_1}\lambda).
\end{align*}
When $0<\lambda<2c_1$,
\begin{align*}
\frac{1}{\tau(e)}\tau(\mathds{1}_{(\lambda,\infty)}(s_c((x-x_n)e)))\leq1<2\exp(-\frac{\ln2}{2c_1}\lambda).
\end{align*}
Therefore, we obtain the desired result by letting
$c={\ln2}/{(2c_1)}$.
\end{proof}

Based on the crude version of Junge/Musat's John-Nirenberg
inequality in Theorem \ref{c-jm-j-n-b} (resp. Corollary
\ref{c-j-n-B}) for $\sf{bmo}(\M)$ (resp. $\mathcal{BMO}(\M)$), the
argument in the proof of Theorem \ref{f-j-n-b} can be adapted to get
the fine version of Junge/Musat's John-Nirenberg inequality.

\begin{corollary}\label{f-jm-j-n-b}
For all $0<p<\infty$, we have
$$\alpha^{-1}_p\|x\|_{\mathsf{bmo}}\leq\mathcal{P}\mathsf{b}_p(x)\leq \beta_p\|x\|_{\mathsf{bmo}},$$ where
\begin{align*}
\mathcal{P}\mathsf{b}_p(x)=\max\{&\sup_{n}\|(dx_n)_n\|_{{\infty}},\quad
\sup_n\sup_{e\in\M_n} \|(x-x_{n})\frac{e}{(\tau(e))^{1/p}}\|_{p},\\
&\sup_n\sup_{e\in\M_n} \|\frac{e}{(\tau(e))^{1/p}}(x-x_{n})\|_{p}\}.
\end{align*}
The constants $\alpha_p$ and $\beta_p$ have the same orders as those
in Theorem \ref{c-j-n-b}.
\end{corollary}

\begin{corollary}\label{f-jm-j-n-B}
For $0<p<\infty$, we have
$$\alpha^{-1}_p\|x\|_{\BMO}\leq\mathcal{PB}_p(x)\leq\beta_p\|x\|_{\BMO},$$
where
\begin{align*}
\mathcal{PB}_p(x)=\max\{
&\sup_n\sup_{e\in\M_n} \|(x-x_{n-1})\frac{e}{(\tau(e))^{1/p}}\|_{p},\\
&\sup_n\sup_{e\in \M_n}
\|\frac{e}{(\tau(e))^{1/p}}(x-x_{n-1})\|_{p}\}.
\end{align*}
The constant $\alpha_p$ and $\beta_p$ have the same orders as those
in Theorem \ref{c-j-n-b}.
\end{corollary}

Again, based on Corollary \ref{f-jm-j-n-B}, by  arguments similar to
the proof of Thoerem \ref{c-weak}, we obtain the exponential
integrability form of the John-Nirenberg inequality for $\BMO(\M)$.

\begin{theorem}\label{weak-B}
Let $x\in \BMO(\M)$. Then for all natural numbers $n\geq1$, all $e\in
\mathcal{P}(\M_n)$ and for all $\lambda>0$, we have
$$\frac{1}{\tau(e)}\tau(\mathds{1}_{(\lambda,\infty)}(|(x-x_{n-1})e|)+\mathds{1}_{(\lambda,\infty)}(|e(x-x_{n-1})|))\leq 4\exp(-\frac{c\lambda}{\|x\|_{\BMO}})$$
with $c$ an absolute constant.
\end{theorem}




\section{atomic decomposition}

\subsection{A crude version of atoms} According to the crude version of the noncommutative John-Nirenberg inequality, we introduce the following

\begin{definition}\label{defi-c-atom}
For $1<q\leq\infty$, $a\in L_1(\M)$ is said to be a $(1,q,c)$-atom
with respect to $(\M_n)_{n\geq1}$, if there exist $n\geq1$ and a
factorization $a=yb$ such that
 \begin{enumerate}[(i)]
 \item $\E_n(y)=0$;
  \item $b\in L_{q'}(\M_n)$ and $\|b\|_{q'}\leq1$;
  \item  $\|y\|_{\mathsf{h}^c_q}\leq1$ for $1<q<\infty$;
$\|y\|_{\mathsf{bmo}^c}\leq1$ for $q=\infty$.
 \end{enumerate}
 Similarly, we define the notion of a $(1,q,r)$-atom with
$a=yb$ replaced by $a=by$.
\end{definition}

\begin{lemma}\label{h1at-to-h1}
Let $1<q\leq\infty$. If $a$ is a $(1,q,c)$-atom, then
$$\|a\|_{\mathsf{h}^c_1}\leq1.$$
The analogous inequality holds for $(1,q,r)$-atoms.
\end{lemma}

\begin{proof}
We first deal with the case $1<q<\infty$. By definition, there
exists an $n$ such that the $(1,q,c)$-atom $a$ admits a
factorization $a=yb$ as in Definition \ref{defi-c-atom}. Then
$$s^2_c(a)=b^*\sum_{k>n}\mathcal{E}_{k-1}|dy_k|^2b=b^*s^2_c(y)b.$$
Thus by H\"{o}lder's inequality,
$$\|a\|_{\mathsf h^c_1}=\|s_c(a)\|_1\leq\|s_c(y)\|_q\|b\|_{q'}\leq1.$$
 For the case $q=\infty$, the calculation is a bit different,
\begin{align*}
\|a\|_{\mathsf{h}^c_1}&=\left\|b^*s^2_c(y)b\right\|^{1/2}_{{1/2}}
=\tau(\mathcal{E}_n(b^*s^2_c(y)b)^{1/2})\\
&\leq\tau((\mathcal{E}_n(b^*s_c(y)b))^{1/2})
\leq\|\mathcal{E}_n(s_c(y))\|_{\infty}\|b\|_1\\
& \leq\left\|y\right\|_{\mathsf{bmo}^c}\|b\|_{1}\leq1.
\end{align*}
We have used the trace preserving property of conditional
expectations in the fourth equality and the operator Jensen
inequality in the first inequality. For the second inequality, we
have used the property that
$\mathcal{E}_n\cdot\mathcal{E}_{k-1}=\mathcal{E}_n$ for all $k>n$
and H\"{o}lder's inequality.
\end{proof}

\begin{definition}

We define $\mathsf{h}^{c}_{1,\mathrm{at}_q}(\M)$ as the Banach space
of all $x\in L_1(\M)$ which admit a decomposition
$x=\sum_k\lambda_ka_k$, where for each $k$, $a_k$ a $(1,q,c)$-atom or
an element in the unit ball of $L_1(\M_1)$, and
$\lambda_k\in\mathbb{C}$ satisfying $\sum_k |\lambda_k|<\infty$. We
equip this space with the norm
$$\|x\|_{\mathsf{h}^{c}_{1,\mathrm{at}_q}}=\inf\sum_k|\lambda_k|,$$ where the infimum is taken over
all decompositions of $x$ described above. Similarly, we define
$\mathsf{h}^{r}_{1,\mathrm{at}_q}(\M)$.
\end{definition}

Now, by Lemma \ref{h1at-to-h1}, we have the obvious inclusion
$\mathsf{h}^c_{1,\mathrm{at}_q}(\M)\subset \mathsf{h}^c_1(\M)$. In
fact, the two spaces coincide thanks to the following theorem.

\begin{theorem}\label{crude h1 isometry}
For all $1<q\le\infty$, we have
$$\mathsf{h}^c_1(\M)=\mathsf{h}^{c}_{1,\mathrm{at}_q}(\M)$$
with equivalent norms. Similarly,
$\mathsf{h}^r_1(\M)=\mathsf{h}^{r}_{1,\mathrm{at}_q}(\M)$ with
equivalent norms.
\end{theorem}

We prove this theorem by duality. We require the following lemmas.

\begin{lemma}\label{embedding}
$(\rm{i})$ For all $1<q\leq2$, $L_2(\M)$ densely and continuously
embeds into $\mathsf h^c_{1,\mathrm{at}_q}(\M)$.

$(\rm{ii})$ For all $2<q\leq\infty$, $L_q(\M)$ densely and
continuously embeds into $\mathsf h^c_{1,\mathrm{at}_q}(\M)$.
\end{lemma}

\begin{proof}
$(\rm{i})$. For any $x\in L_2(\M)$, we decompose it as a linear
combination of two atoms:
$$x=\|x-\mathcal{E}_1(x)\|_2\frac{x-\mathcal{E}_1(x)}{\|x-\mathcal{E}_1(x)\|_2}+\|\mathcal{E}_1
(x)\|_2\frac{\mathcal{E}_1(x)}{\|\mathcal{E}_1(x)\|_2}.$$
 Indeed, on the
one hand, ${\mathcal{E}_1(x)}/{\|\mathcal{E}_1(x)\|_2}\in
L_2(\M_1)\subset L_1(\M_1)$ and
$$\|\frac{\mathcal{E}_1(x)}{\|\mathcal{E}_1(x)\|_2}\|_1
=\frac{\|\mathcal{E}_1(x)\|_1}{\|\mathcal{E}_1(x)\|_2}\leq1.$$ On
the other hand,
$$\frac{x-\mathcal{E}_1(x)}{\|x-\mathcal{E}_1(x)\|_2}
=\frac{x-\mathcal{E}_1(x)}{\|x-\mathcal{E}_1(x)\|_2}\cdot\mathds{1}\doteq
y\cdot b.$$ Clearly, $\mathcal{E}_1(y)=0$, $\|b\|_{q'}\leq1$ and
$$\|y\|_{\mathsf h^c_q}=\|\frac{x-\mathcal{E}_1(x)}{\|x-\mathcal{E}_1(x)\|_2}\|_{\mathsf h^c_q}
\leq\|\frac{x-\mathcal{E}_1(x)}{\|x-\mathcal{E}_1(x)\|_2}\|_{\mathsf
h^c_2}\leq1.$$ Thus $x$ is a sum of two atoms and
$$\|x\|_{\mathsf h^c_{1,\mathrm{at}_q}}\leq\|x-\mathcal{E}_1(x)\|_2+\|\mathcal{E}_1
(x)\|_2\leq\sqrt{2}\|x\|_2.$$ The density is trivial.

$(\rm{ii})$. This case is similar to the previous one. We first deal
with the case $2<q<\infty$. Given $x\in L_q(\M)$, we write again:
$$x=c_q\|x-\mathcal{E}_1(x)\|_q\frac{x-\mathcal{E}_1(x)}{c_q\|x-\mathcal{E}_1(x)\|_q}+\|\mathcal{E}_1
(x)\|_q\frac{\mathcal{E}_1(x)}{\|\mathcal{E}_1(x)\|_q},$$ where
$c_q$ is fixed below. Indeed,
${\mathcal{E}_1(x)}/{\|\mathcal{E}_1(x)\|_q}\in L_q(\M_1)\subset
L_1(\M_1)$ and
$$\|\frac{\mathcal{E}_1(x)}{\|\mathcal{E}_1(x)\|_q}\|_1
=\frac{\|\mathcal{E}_1(x)\|_1}{\|\mathcal{E}_1(x)\|_q}\leq1.$$ On
the other hand,
$$\frac{x-\mathcal{E}_1(x)}{c_q\|x-\mathcal{E}_1(x)\|_q}=\frac{x-\mathcal{E}_1(x)}{c_q\|x-\mathcal{E}_1(x)\|_q}\cdot\mathds{1}\doteq y\cdot b,$$
$$\mathcal{E}_1(\frac{x-\mathcal{E}_1(x)}{c_q\|x-\mathcal{E}_1(x)\|_q})=0,\quad\|b\|_{q'}\leq1$$
and the noncommutative Burkholder inequality in \cite{JuXu03} yields
$$\|y\|_{\mathsf h^c_q}=\|\frac{x-\mathcal{E}_1(x)}{c_q\|x-\mathcal{E}_1(x)\|_q}\|_{\mathsf h^c_q}
\leq
c_q\|\frac{x-\mathcal{E}_1(x)}{c_q\|x-\mathcal{E}_1(x)\|_q}\|_{q}\leq
1.$$ Therefore,
$$\|x\|_{\mathsf h^c_{1,\mathrm{at}_q}}\leq c_q\|x-\mathcal{E}_1(x)\|_q+\|\mathcal{E}_1
(x)\|_q\leq (2c_q+1)\|x\|_q.$$
 The case $q=\infty$ is proved in the
same way just by replacing the noncommutative Burkholder inequality by
the trivial fact that
$\|\cdot\|_{\mathsf{bmo}^c}\leq\|\cdot\|_{\infty}$. The density is
trivial.
\end{proof}

\begin{lemma}\label{duality}
Let $1<q<\infty$. Then
$$(\mathsf{h}^c_{1,\mathrm{at}_{q}}(\M))^*=\mathsf{bmo}^c_{q'}(\M)$$
with equivalent norms. More precisely,
 \begin{enumerate}[\rm(i)]
 \item Every $x\in\mathsf{bmo}^c_{q'}(\M)$ defines a bounded
linear functional on $\mathsf{h}^c_{1,\mathrm{at}_{q}}(\M)$ by
 \begin{align}\label{c-d-b}
\varphi_x(a)=\tau(x^*a), \forall a\in(1,q,c)\textrm{-atoms}.
\end{align}

\item Conversely, each $\varphi\in
(\mathsf{h}^c_{1,\mathrm{at}_{q}}(\M))^*$ is given as (\ref{c-d-b})
by some $x\in\mathsf{bmo}^c_{q'}(\M)$.
 \end{enumerate}
 Similarly,
$(\mathsf{h}^r_{1,\mathrm{at}_{q}}(\M))^*=\mathsf{bmo}^r_{q'}(\M)$
with equivalent norms.
\end{lemma}

\begin{proof}
$\rm(i)$ Let $x\in \mathsf{bmo}^c_{q'}$, and $a=yb$ where $a$ is a
$(1,q,c)$-atom as in Definition \ref{defi-c-atom}. Then
\begin{align*}
|\tau(x^*a)|&=|\tau(\mathcal{E}_{n}(x^*y)b)|\\&=|\tau(\mathcal{E}_{n}((x^*-x^*_n)y)b)|
=|\tau(((x-x_n)b^*)^*y)|.
\end{align*}
Thus, by the duality identity $\mathsf h^c_q(\M)=(\mathsf h^c_{q'}(\M))^*$ (see
\cite{JuXu03} for the relevant constants),
\begin{align*}
|\tau(x^*a)|\leq \left\|(x-x_n)b^*\right\|_{\mathsf
h^c_{q'}}\|y\|_{\mathsf{h}^c_q} &\leq \|x\|_{\mathsf{bmo}^c_{q'}}.
\end{align*}

$\rm(ii)$. Let $\varphi$ be any linear functional on
$\mathsf{h}^c_{1,\mathrm{at}_q}(\M)$. When $1<q\leq2$, by Lemma
\ref{embedding}
we can find $x\in L_2(\M)$ such that
$$\varphi(y)=\tau(x^*y),\qquad\forall y\in L_2(\M),$$
and
$$\|\varphi\|=\sup_{y\in L_2,\|y\|_{\mathsf{h}^c_{1,\mathrm{at}_q}}\leq1}|\tau(x^*y)|.$$
When $2<q<\infty$, by the same Lemma \ref{embedding}, we get the
same representation of $\varphi$ with an  $x\in L_{q'}(\M)$. Then
fix $n$ and take any $b\in \M_n$ with $\|b\|_{q'}\leq1$. Again, by
the duality $\mathsf h^c_q(\M)=(\mathsf h^c_{q'}(\M))^*$, we do the
following calculation:
\begin{align*}
\|(x-x_n)b\|_{\mathsf{h}^c_{q'}}&=\sup_{\|y\|_{(\mathsf{h}^c_{q'})^*}\leq 1}|\tau(b^*(x^*-x^*_n)y)|\\
&\leq\sup_{\|y\|_{\mathsf{h}^c_q}\leq cq}|\tau(b^*(x^*-x^*_n)y)|\\
&=\sup_{\|y\|_{\mathsf{h}^c_q}\leq cq}|\tau((x^*-x^*_n)(y-y_n)b^*)|\\
&=\sup_{\|y\|_{\mathsf{h}^c_q}\leq cq}|\tau(x^*((y-y_n)b^*))|\\
&\leq cq\|\varphi\|
\end{align*}
Here, we have used the fact that $\tau(x-x_n)=\tau(y-y_n)=0$ in the
second and third equality respectively. The second inequality is
due to the fact that $(y-y_n)b^*$ is a $(1,q,c)$-atom.
\end{proof}

Now we are at a position to prove Theorem \ref{crude h1 isometry}.

\begin{proof}
We consider here only the case $1<q<\infty$ and postpone the case $q=\infty$ to the end of the proof of Theorem \ref{fine h1
isometry} below. We only need to show the inclusion
$$\mathsf{h}^c_{1}(\M)\subset \mathsf{h}^c_{1,\mathrm{at}_q}(\M).$$
Take $x\in \mathsf h^c_{1,\mathrm{at}_q}(\M)$, by Theorem
\ref{c-j-n-b} and Lemma \ref{duality}, we can conduct the
following calculation,
\begin{align*}
\|x\|_{\mathsf{h}^c_{1,\mathrm{at}_q}}&=\sup_{\|y\|_{(\mathsf{h}^c_{1,\mathrm{at}_q})^*}\leq1}|\tau(x^*y)|\\
&\leq\sup_{\|y\|_{\mathsf{bmo}^c_{q'}}\leq cq}|\tau(x^*y)|\\
&\leq\sup_{\|y\|_{\mathsf{bmo}^c}\leq cq}|\tau(x^*y)|
\leq cq\|x\|_{\mathsf{h}^c_{1}}.
\end{align*}
Then we end the proof with the density of $\mathsf
h^c_{1,\mathrm{at}_q}(\M)$ in $\mathsf h^c_1(\M)$.
\end{proof}

\begin{definition}
We define
$$\mathsf{h}_{1,\mathrm{at}_{q}}(\M)=\mathsf{h}^c_{1,\mathrm{at}_{q}}(\M)+\mathsf{h}^r_{1,\mathrm{at}_{q}}(\M)+\mathsf{h}^d_1(\M)$$
equipped with the sum norm
$$\|x\|_{\mathsf{h}_{1,\mathrm{at}_{q}}}=\inf_{x=x_c+x_r+x_d}\{\|x_c\|_{\mathsf{h}^c_{1,\mathrm{at}_{q}}}
+\|x_r\|_{\mathsf{h}^r_{1,\mathrm{at}_{q}}}+\|x_d\|_{\mathsf{h}^d_1}\}.$$
\end{definition}

Then by Theorem \ref{crude h1 isometry}, we obtain the atomic
decomposition of $\mathsf{h}_1(\M)$.

\begin{corollary}
We have
$$\mathsf{h}_1(\M)=\mathsf{h}_{1,\mathrm{at}_{q}}(\M)$$
with equivalent norms.
\end{corollary}

Combined with Davis' decomposition presented in \cite{Per09}, the
above theorem yields
$\mathcal{H}_1(\M)=\mathsf{h}_{1,\mathrm{at}_{q}}(\M)$ with
equivalent norms. In other words, we obtain an atomic decomposition
for $\mathcal{H}_1(\M)$ too.

\subsection{A fine version of atoms}

\begin{definition}\label{defi-f-atom}
For $1<q\leq\infty$, $a\in L_1(\M)$ is said to be a
$(1,q,c)_{\mathrm{pr}}$-atom with respect to $(\M_n)_{n\geq1}$, if
there exist $n\geq1$ and a projection $e\in\mathcal{P}(\M_n)$ such
that
 \begin{enumerate}[(i)]
 \item $\E_n(a)=0$;

\item$r(a)\leq e$;

\item $\|a\|_{\mathsf h^c_q}\leq(\tau(e))^{-\frac{1}{q'}}$
for $1<q<\infty$; $\|a\|_{\mathsf{bmo}^c}\leq{(\tau(e))}^{-1}$ for
$q=\infty$.
  \end{enumerate}
Similarly, we define $(1,q,r)_{\mathrm{pr}}$-atoms with
$r(a)$ replaced by $l(a)$.
\end{definition}

\begin{remark}\label{f-is-c}
A $(1,q,c)_{\mathrm{pr}}$-atom $a$ is necessarily a $(1,q,c)$-atom.
Indeed, we can factorize $a$ as $a=yb$ with
$y=a(\tau(e))^{{1}/{q'}}$ and $b=e(\tau(e))^{-1/{q'}}$.
\end{remark}

\begin{definition}
We define $\mathsf{h}^{c}_{1,\mathrm{at}_{q,\mathrm{pr}}}(\M)$ to be
the Banach space of all $x\in L_1(\M)$ which admit a decomposition
$x=\sum_k\lambda_ka_k$, where for each $k$, $a_k$ is a
$(1,q,c)_{\mathrm{pr}}$-atom or an element in the unit ball of
$L_1(\M_1)$, and $\lambda_k\in\mathbb{C}$ satisfying $\sum_k
|\lambda_k|<\infty$. We equip this space with the norm
$$\|x\|_{\mathsf{h}^{c}_{1,\mathrm{at}_{q,\mathrm{pr}}}}=\inf\sum_k|\lambda_k|,$$ where the infimum is taken over
all decompositions of $x$ described above. Similarly, we define
$\mathsf{h}^{r}_{1,\mathrm{at}_{q,\mathrm{pr}}}(\M)$.
\end{definition}

Now, by Remark \ref{f-is-c} and Lemma \ref{crude h1 isometry}, we
have the obvious inclusion
$\mathsf{h}^c_{1,\mathrm{at}_{q,\mathrm{pr}}}(\M)\subset
\mathsf{h}^c_{1}(\M)$. In fact, the two spaces coincide thanks to
the following theorem.

\begin{theorem}\label{fine h1 isometry}
For all $1<q\leq\infty$, we have
$$\mathsf{h}^c_1(\M)=\mathsf{h}^{c}_{1,\mathrm{at}_{q,\mathrm{pr}}}(\M)$$
with equivalent norms. Similarly,
$\mathsf{h}^r_1(\M)=\mathsf{h}^{r}_{1,\mathrm{at}_{q,\mathrm{pr}}}(\M)$
with equivalent norms.
\end{theorem}

Again, we prove this theorem for $1<q<\infty$ by showing
$(\mathsf{h}^c_{1,\mathrm{at}_{q,\mathrm{pr}}}(\M))^*=\mathsf{bmo}^c_{q',\mathrm{pr}}(\M)$.
The latter duality equality is proved in the same way as Theorem
\ref{duality}. We leave the details to the reader. However by the
argument in Theorem 4.6, we can not prove the theorem in the case
$q=\infty$, due to the lack of Riesz representation. Here we provide
another way to do it, which seems new, even in the commutative case.

Let $\mathcal{P}$ be the set of projections of $\M$. Given
$e\in\mathcal{P}$ let
 $$n_e=\min\{k\;:\;e\in\mathcal{P}(\mathcal{M}_k)\}.$$
Note that $n_e=\infty$ if the set on the right hand side is empty. This case is of no interest in the discussion below. For a family $(g_e)_{e\in\mathcal{P}}\subset \mathsf{bmo}^c(\M)$ define
 $$\|(g_e)_e\|_{L^{\mathcal{P}}_1(\mathsf{bmo}^c)}=\sum_{e\in\mathcal{P}}\tau(e)\|g_e\|_{\mathsf{bmo}^c}.$$
We will consider the Banach space:
 $$
L^{\mathcal{P}}_1(\mathsf{bmo}^c)=\{(g_e)_e\;:\; g_ee=g_e,\;
\mathcal{E}_{n_e}g_e=0,\;
\|(g_e)_e\|_{L^{\mathcal{P}}_1(\mathsf{bmo}^c)}<\infty\}.$$ We will
also need the following space consisting of families in
$\mathsf{h}^c_1(\M)$:
 $$L^{\mathcal{P}}_\infty(\mathsf{h}^c_1)=\{(f_e)_e\;: \;f_ee=f_e,\;
 \mathcal{E}_{n_e}f_e=0,\; \|(f_e)_e\|_{L^{\mathcal{P}}_{\infty}(\mathsf{h}^c_1)}<\infty\},$$
where
 $$\|(f_e)_e\|_{L^{\mathcal{P}}_{\infty}(\mathsf{h}^c_1)}=\sup_{e\in\mathcal{P}}\frac{1}{\tau(e)}\,\|f_e\|_{\mathsf{h}_1^c}.$$
For convenience, we denote $L^{\mathcal{P}}_1(\mathsf{bmo}^c)$ by
$X$ and $L^{\mathcal{P}}_{\infty}(\mathsf{h}^c_1)$ by $Z$. We embed
$\mathsf{bmo}^c_{1,\mathrm{pr}}(\M)$ isomorphically into $Z$ via the
following  map
 $$\pi(y)=((y-y_{n_e})e)_e.$$
Set $Y=\pi(\mathsf{bmo}^c_{1,\mathrm{pr}}(\M))$.

\begin{lemma}
With the notation above we have
 \begin{enumerate}[\rm(i)]
 \item $Z$ is a subspace of $X^*$ with equivalent norms, so is $Y$.
 \item $Y$ is w*-closed in $X^*$.
 \end{enumerate}
\end{lemma}

\begin{proof}
(i). Let $(f_e)_e\in Z$, for any $(g_e)_e\in X$, we have
\begin{align*}
|\langle (f_e)_e,(g_e)_e\rangle|&=|\sum_e\tau((f_e)^*g_e)|\\
&\leq\sqrt{2}\sum_e\|f_e\|_{\mathsf{h}^c_1}\|g_e\|_{\mathsf{bmo}^c}\\
&\leq\sqrt{2}\sup_e\frac{1}{\tau(e)}\,\|f_e\|_{\mathsf{h}_1^c}\cdot\sum_e\tau(e)\,\|g_e\|_{\mathsf{bmo}^c}\\
&=\sqrt{2}\|(f_e)_e\|_{Z}\|(g_e)_e\|_{X}.
\end{align*}
Thus we get $\|(f_e)_e\|_{X^*}\leq\sqrt{2}\|(f_e)_e\|_{Z}.$

We turn to the proof of the inverse inequality. For any $(f_e)_e\in
Z$, fix $e_0\in\mathcal{P}$, we have
\begin{align*}
\frac{1}{\tau({e_0})}\,\|f_{e_0}\|_{\mathsf
h^c_1}&=\sup_{\|g\|_{\mathsf{bmo}^c}\leq1}\frac{1}{\tau({e_0})}\,\big|\tau((f_{e_0})^*g)\big|\\
&=\sup_{\|g\|_{\mathsf{bmo}^c}\leq1}\frac{1}{\tau({e_0})}\,\big|\tau((f_{e_0})^*(g-g_{n_{e_0}}){e_0})\big|\\
&\leq\sup_{\|(g-g_{n_{e_0}}){e_0}\|_{\mathsf{bmo}^c}\leq1}\frac{1}{\tau({e_0})}\,\big|\tau((f_{e_0})^*(g-g_{n_{e_0}}){e_0})\big|.
\end{align*}
Then we define $(g_e)_e$ as
$g_e={(g-g_{n_{e_0}}{e_0})}/{\tau({e_0})}$ if $e={e_0}$, otherwise
$g_e=0$. Thus
\begin{align*}
\frac{1}{\tau({e_0})}\,\|f_{e_0}\|_{\mathsf h^c_1}\leq\|(f_e)_e\|_{X^*}\|(g_e)_e\|_X\leq\|(f_e)_e\|_{X^*},
\end{align*}
which implies $\|(f_e)_e\|_{Z}\leq\|(f_e)_e\|_{X^*}$.

(ii). Since $Y$ is a subspace of $X^*$, by Krein and Smulian's
theorem, we only need to prove that for all $t>0$, $Y\cap B_t(X^*)$
is w*-closed in $X^*$, where $B_t(X^*)$ is the closed ball of $X^*$ centered at the origin and with radius $t$.
 Take a net $(y^{\alpha})_{\alpha}\subset
\mathsf{bmo}^c_{1,\mathrm{pr}}(\M)$ such that
$\pi((y^{\alpha})_{\alpha})\subset Y\cap B_t(X^*)$. Hence
$(y^{\alpha})_{\alpha}$ are bounded in
$\mathsf{bmo}^c_{1,\mathrm{pr}}(\M)$. Suppose that,
\begin{align}\label{w*cvinX*}
\langle \pi(y^{\alpha}),(g_e)_e\rangle\rightarrow
\langle\xi,(g_e)_e\rangle,\quad\quad \forall (g_e)_e\in X,
\end{align}
for some $\xi\in B_t(X^*)$.
 We will show that $\xi\in Y$, which will complete
 the proof. We need two facts. The first one is that
 $\mathsf{bmo}^c_{1,\mathrm{pr}}(\M)$ is a dual space by Theorem \ref{f-j-n-b}, so its unit ball is w*-compact.
 Therefore, the bounded net $(y^{\alpha})_{\alpha}$
 in $\mathsf{bmo}^c_{1,\mathrm{pr}}(\M)$ admits a $w^*$-cluster point $y$. Without loss of generality, we assume that $(y^{\alpha})_{\alpha}$ converges to $y$ in the $w^*$-topology:
\begin{align}\label{w*cvinbmoc1pr}
\langle y^{\alpha},x\rangle\rightarrow \langle
y,x\rangle,\quad\quad \forall x\in \mathsf{h}^c_1(\M).
\end{align}
The second fact is that for any $(g_e)_e\in X$, the sum $\sum_eg_e$
is absolutely summable in $\mathsf{h}^c_1(\M)$. Indeed, by Lemma \ref{h1at-to-h1}
\begin{align*}
 \sum_{e}\|g_e\|_{\mathsf{h}^c_1}\leq\sum_{e}\tau(e)\|g_e\|_{\mathsf{bmo}^c}=\|(g_e)_e\|_{X}.
\end{align*}
Therefore, for any $(g_e)_e\in X$, we have
\begin{align*}
\langle
\pi(y^{\alpha}),(g_e)_e\rangle&=\sum_e\tau(((y^{\alpha}_e-y^{\alpha}_{n_e})e)^*g_e)\\
&=\tau((y^{\alpha})^*\sum_eg_e)
\end{align*}
Combining \ref{w*cvinX*} and \ref{w*cvinbmoc1pr},
we deduce that $\xi=\pi(y)\in Y$, as desired.
\end{proof}

We can now prove Theorem \ref{fine h1
isometry} in the case of $q=\infty$.

\begin{proof}
 Let  $Y_{\perp}$ be the preannihilator of
$Y$ in $X^*$:
 $$Y_{\perp}=\{(g_e)_e\in X\;:\; \langle\pi(y),(g_e)_e\rangle=0,\forall
y\in\mathsf{bmo}^c_{1,\mathrm{pr}}(\M)\}.$$
Then by the bipolar theorem
 $$Y\simeq (X/Y_{\perp})^*.$$
Using the second fact in the proof of the previous lemma, we get
\begin{align*}
Y_{\perp}&=\{(g_e)_e\in X\;:\; \tau(y^*\sum_eg_e)=0,\forall
y\in\mathsf{bmo}^c_{1,\mathrm{pr}}(\M)\}\\
&=\{(g_e)_e\in X\;:\; \sum_eg_e=0\; \mathrm{in}
\;\mathsf{h}^c_1(\M)\}.
\end{align*}
Then for $(g_e)_e\in X/{Y_{\perp}}$, let
 $$g=\sum_{e\in\mathcal P}g_e.$$
Then
\begin{align*}
\|(g_e)_e\|_{X/Y_{\perp}}
&=\inf\{\sum_e\tau(e)\,\|(g'_e)_e\|_{\mathsf{bmo}^c}\,:\, g=\sum_eg'_e,\; (g'_e)_e\in X\}\\
&=\inf\{\sum_e|\lambda_e|\,:\,g=\sum_e\lambda_ea_e,\,(\lambda_ea_e)_e\in X,\, \|a_e\|_{\mathsf{bmo}^c}\le\frac{1}{\tau(e)}\}\\
&=\|g\|_{\mathsf{h}^c_{1,\mathrm{at}_{\infty,\mathrm{pr}}}}.
\end{align*}
Consequently, for any
$x\in\mathsf{h}^c_{1,\mathrm{at}_{\infty,\mathrm{pr}}}(\M)$ and any
decomposition $x=\sum_e\lambda_ea_e$,
\begin{align*}
\|x\|_{\mathsf{h}^c_{1,\mathrm{at}_{\infty,\mathrm{pr}}}}&=\|(\lambda_ea_e)_e\|_{X/Y_{\perp}}\\
&=\|(\lambda_ea_e)_e\|_{Y^*}\\
&=\sup_{y\in\mathsf{bmo}^c_{1,\mathrm{pr}},\|\pi(y)\|_{Y}\leq1}|\langle(\lambda_ea_e),\pi(y)\rangle|\\
&\leq\sup_{\|y\|_{\mathsf{bmo}^c}\leq
c}|\tau((\sum_e\lambda_ea_e)^*y)|\leq c\|x\|_{\mathsf{h}^c_1}.
\end{align*}
Therefore, combined with Lemma \ref{h1at-to-h1} and Remark
\ref{f-is-c}, the density of
$\mathsf{h}^c_{1,\mathrm{at}_{\infty,\mathrm{pr}}}(\M)$ in
$\mathsf{h}^c_1(\M)$ (due to Lemma \ref{embedding}) yields the desired duality identity
$\mathsf{h}^c_{1,\mathrm{at}_{\infty,\mathrm{pr}}}(\M)=\mathsf{h}^c_1(\M)$.
\end{proof}

Let us return back to the unsettled case $q=\infty$ in the proof of Theorem \ref{crude h1 isometry}.  Since a fine atom  is necessarily a crude
atom, we get
$\mathsf{h}^c_1(\M)\subset\mathsf{h}^c_{1,\mathrm{at}_{\infty}}(\M)$,
hence $\mathsf{h}^c_1(\M)=\mathsf{h}^c_{1,\mathrm{at}_{\infty}}(\M)$
with equivalent norms due to Lemma \ref{h1at-to-h1}. Thus Theorem \ref{crude h1 isometry}
is completely proved.


\begin{definition}
We define
$$\mathsf{h}_{1,\mathrm{at}_{q,\mathrm{pr}}}(\M)=\mathsf{h}^c_{1,\mathrm{at}_{q,\mathrm{pr}}}(\M)+\mathsf{h}^r_{1,\mathrm{at}_{q,\mathrm{pr}}}(\M)+\mathsf{h}^d_1(\M)$$
equipped with the sum norm
$$\|x\|_{\mathsf{h}_{1,\mathrm{at}_{q,\mathrm{pr}}}}=\inf_{x=x_c+x_r+x_d}\{\|x_c\|_{\mathsf{h}^c_{1,\mathrm{at}_{q,\mathrm{pr}}}}+\|x_r\|_{\mathsf{h}^r_{1,\mathrm{at}_{q,\mathrm{pr}}}}+\|x_d\|_{\mathsf{h}^d_1}\}.$$
\end{definition}

Then by Theorem \ref{fine h1 isometry} and Perrin's noncommutative
Davis decomposition (see \cite{Per09}), we get the atomic
decomposition of $\mathsf{h}_1(\M)$ and $\mathcal{H}_1(\M)$.

\begin{corollary}
We have
$$\mathcal{H}_1(\M)=\mathsf{h}_1(\M)=\mathsf{h}_{1,\mathrm{at}_{q,\mathrm{pr}}}(\M),$$ for any $1<q\leq\infty$,
with equivalent norms.
\end{corollary}

However, using Corollary \ref{f-jm-j-n-b}, we can obtain another
kind of atomic decomposition for $\mathsf{h}_1(\M)$ or
$\mathcal{H}_1(\M)$, which is exactly the noncommutative analogue of
the classical case.
\begin{definition}\label{defi-atom}
For $1<q\leq\infty$, $a\in L_1(\M)$ is said to be a $(1,q)$-atom
with respect to $(\M_n)_{n\geq1}$, if there exist $n\geq1$ and a
projection $e\in\mathcal{P}(\M_n)$ such that
\begin{enumerate}[\rm(i)]
\item $\E_n(a)=0$;
\item $r(a)\leq e$ or $l(a)\leq e$;
\item $\|a\|_{q}\leq(\tau(e))^{-\frac{1}{q'}}$.
\end{enumerate}
\end{definition}

\begin{definition}
We define $\mathsf{h}^{\mathrm{at}}_{1,q}(\M)$ as the Banach space
of all $x\in L_1(\M)$ which admit a decomposition
$x=y+\sum_k\lambda_ka_k$, where for each $k$, $a_k$ is a $(1,q)$-atom or
an element in the unit ball of $L_1(\M_1)$, $\lambda_k\in\mathbb{C}$
satisfying $\sum_k |\lambda_k|<\infty$, and where the martingale differences of $y$ satisfy
$\sum_{j\geq1}\|dy_j\|_1<\infty$. We equip this space with the norm
$$\|x\|_{\mathsf{h}^{\mathrm{at}}_{1,q}}=\inf\big\{\sum_j\|dy_j\|_1+\sum_k|\lambda_k|\big\},$$
where the infimum is taken over all decompositions of $x$ as
above.
\end{definition}

\begin{lemma}\label{hat-to-h1}
If $a$ is a $(1,q)$-atom, then
$$\|a\|_{\mathsf{h}_1}\leq \frac{cq}{q-1}.$$
\end{lemma}

\begin{proof}

Without loss of generality, suppose $a$ is a $(1,q)$-atom with
$r(a)\leq e$. We apply Corollary 3.18 and the duality
$(\mathsf{h}_1(\M))^*=\mathsf{bmo}(\M)$.
\begin{align*}
\|a\|_{\mathsf{h}_1}&\leq c\sup_{\|x\|_{\mathsf{bmo}}\leq1}\tau (x^*a)\\
&=c\sup_{\|x\|_{\mathsf{bmo}}\leq1}\tau ((x-x_n)^*a)\\
&=c\sup_{\|x\|_{\mathsf{bmo}}\leq1}\tau (((x-x_n)e)^*a)\\
&\leq c\|a\|_q\|(x-x_n)e\|_{q'}\leq cq'.
\end{align*}
\end{proof}

\begin{theorem}\label{h1 isometry}
For all $1<q\leq\infty$, we have
$$\mathcal{H}_1(\M)=\mathsf{h}_1(\M)=\mathsf{h}^{\mathrm{at}}_{1,q}(\M)$$
with equivalent norms.
\end{theorem}

By Lemma \ref{hat-to-h1}, Corollary \ref{f-jm-j-n-b} and using
arguments similar to those in the proof of Theorem \ref{crude h1
isometry}, we can prove the theorem for the case $1<q<\infty$. For
the case $q=\infty$, we use the argument in Theorem \ref{fine h1
isometry}. Instead of $L^{\mathcal{P}}_1(\mathsf{bmo}^c)$ and
$L^{\mathcal{P}}_\infty(\mathsf{h}^c_1)$, we consider the following
two spaces:
\begin{align*}
 L^{\mathcal{P}}_1(L_{\infty})&=\big\{(g_e)_e\;: \; g_ee=g_e\,
\mathrm{or}\, eg_e=g_e,
\mathcal{E}_{n_e}g_e=0,\, \|(g_e)_e\|_{L^{\mathcal{P}}_1(L_{\infty})}<\infty\big\},\\
L^{\mathcal{P}}_\infty(L_1)&=\big\{(f_e)_e\;: \; f_ee=f_e\,
\mathrm{or}\, ef_e=f_e,\, \mathcal{E}_{n_e}f_e=0,\,  \|(f_e)_e\|_{L^{\mathcal{P}}_{\infty}(L_1)}<\infty\big\},
\end{align*}
where
\begin{align*}
\|(g_e)_e\|_{L^{\mathcal{P}}_1(L_{\infty})}&=\sum_e\tau(e)\,\|g_e\|_{{\infty}},\\
 \|(f_e)_e\|_{L^{\mathcal{P}}_{\infty}(L_1)}&=\max\big\{\sup_e\frac{1}{\tau(e)}\,\|f_ee\|_{1},\;
\sup_e\frac{1}{\tau(e)}\,\|ef_e\|_{1}\big\}.
 \end{align*}
 Then by Lemma
\ref{hat-to-h1} and Corollary \ref{f-jm-j-n-b}, we get the announced
results. We leave the details to the reader.

\begin{remark} The part of this paper on the crude versions of the John-Nirenberg inequalities and atomic decomposition can be easily extended to the type III case  with minor modifications.
\end{remark}

\section{An open question of Junge and Musat}

It is an open question asked in \cite{JuMu07} (on page 136) that given
$2<p<\infty$, whether there exists a constant $c_p$ such that
\begin{equation}
\sup_k\|\mathcal{E}_k|x-\mathcal{E}_{k-1}x|^p\|_\infty ^{\frac
1p}\leq c_p\|x\|_{\BMO}? \label{jm}
\end{equation}

It is easy to see that the answer is negative for matrix-valued
functions with irregular filtration. In the following, we show that
the answer is negative even for matrix-valued dyadic martingales.
Recall that Remark \ref{counter example2} already shows that the answer is negative if one considers the column norm $\|\cdot\|_{\BMO^c}$ alone on the right hand side.

Let $\M$ and $\M_k$ be as in Remark \ref{counter example2}.   We
consider this special case and show that the best constant
$c_p(n)\;$such that (\ref{jm}) holds is bigger than $c(\log
(n+1))^{1/p}$ for all $p\geq 3.$  Let $b$ be an $M_n$-valued
function on ${\Bbb {T}}$. We need the so-called ``sweep'' function
of $b$
\[
S(b)=\sum_{k=1}^\infty |db_k|^2.
\]
Note that it is just the square of the usual square function.
Matrix-valued sweep functions have been studied in \cite{BlPo08},
\cite{GPTV}, \cite{Mei06} etc. It is proved in \cite{Mei06} that the
best constant $c_n$ such that
\begin{equation}
\|S(b)\|_{\BMO^c}\leq c_n\|b\|_\infty ^2  \label{swm}
\end{equation}
is $c(\log (n+1))^2$. A similar result had been proved previously by
Blasco and Pott (see \cite{BlPo08}) by considering
$\|b\|_{\BMO^c}^2$ on the right side of (\ref{swm}).

\begin{lemma}\label{negativeanswer1}
Assume $\|f\|_{\BMO^c}\leq
c(n)\sup_k\|\mathcal{E}_k|f-\mathcal{E}_{k-1}f\||_\infty $ for any
selfadjoint $f$. Then $c(n)\geq c(\log (n+1))^2.$
\end{lemma}

\begin{proof}
Under the assumption, we have
\begin{align*}
\|S(b)\|_{\BMO^c} &\leq c(n)\sup_m\|\mathcal{E}_m|S(b)-\mathcal{E}_{m-1}S(b)|dt\|_\infty \\
&=c(n)\sup_m\Big\|\mathcal{E}_m\big| \sum_{k=1}^\infty
|db_k|^2-\mathcal{E}_{m-1}\sum_{k=1}^\infty
|db_k|^2\big| \Big\|_\infty \\
&=c(n)\sup_m\Big\|\mathcal{E}_m\big| \sum_{k=m}^\infty
|db_k|^2-\mathcal{E}_{m-1}\sum_{k=m}^\infty |db_k|^2\big|
\Big\|_\infty.
\end{align*}
Let $x=\sum_{k=m}^\infty |db_k|^2$ and
$y=\mathcal{E}_{m-1}\sum_{k=m}^\infty |db_k|^2$. By the convexity of
$|\cdot|^2$, we get
$$\big|\frac{x-y}{2}\big|^2\leq\frac{|x|^2+|y|^2}{2}\leq\frac{|x|^2+\|y\|_{\infty}^2\mathds{1}}{2}
\leq\frac{(|x|+\|y\|_{\infty}\mathds{1})^2}{2}.$$ Then by
L\"owner-Heinz's inequality,
$$\big|\frac{x-y}{2}\big|\leq\frac{|x|+\|y\|_{\infty}\mathds{1}}{\sqrt 2}.$$
Thus by the triangle inequality, we have
\begin{align*}
\|S(b)\|_{\BMO^c}&\leq
2c(n)\sup_m\left\|\mathcal{E}_mx+\|y\|_{\infty}\mathds{1}\right\|_\infty \\
&=2c(n)\sup_m\|\mathcal{E}_m|b-\mathcal{E}_{m-1}b|^2\|_\infty
+2c(n)\|\mathcal{E}_{m-1}|b-\mathcal{E}_{m-1}b|^2\|_\infty
\\
&\leq 2c(n)\|b\|_{\BMO^c}^2+2c(n)\|\mathcal{E}_m|b-\mathcal{E}_{m-1}b|^2\|_\infty \\
&\leq 4c(n)\|b\|_{\BMO^c}^2.
\end{align*}
We then get $c(n)\geq c(\log (n+1))^2$ by (\ref{swm}).
\end{proof}

\begin{lemma}\label{holderforexpectation}
Let $0<p<\infty$ and $\mathcal{E}_m$ be the conditional expectation
from ${\M}$ onto ${\M}_m $, we have
\[
\|\mathcal{E}_m|x|^{\frac{p+1}2}\|_\infty \leq
\|\mathcal{E}_m|x|^p\|_\infty ^{\frac 12}\|\mathcal{E}_m|x\||_\infty
^{\frac 12}.
\]
\end{lemma}

\begin{proof} By H\"{o}lder's inequality, we get
\begin{eqnarray*}
\|\mathcal{E}_m|x|^{\frac{p+1}2}\|_\infty
&=&\sup_{\|a\|_{L_1^{+}({\M}_m)\leq
1}}\tau (\mathcal{E}_m|x|^{\frac{p+1}2}a) \\
&=&\sup_{\|a\|_{L_1^{+}({\M}_m)\leq 1}}\tau (a^{\frac 12}|x|^{\frac
p2}|x|^{\frac 12}a^{\frac 12}) \\
&\leq &\sup_{\|a\|_{L_1^{+}({\M}_m)\leq 1}}(\tau (a|x|^p))^{\frac
12}(\tau
(a|x|))^{\frac 12} \\
&=&\|\mathcal{E}_m|x|^p\|_\infty ^{\frac
12}\|\mathcal{E}_m|x\||_\infty ^{\frac 12}.
\end{eqnarray*}
\end{proof}

\begin{theorem}\label{negativeanswer2}
Suppose $\sup_k\|\mathcal{E}_k|f-\mathcal{E}_{k-1}f|^p\|_\infty
^{1/p}\leq c_p(n)\|f\|_{\BMO}$ for some $p\geq 3.$ Then
\[
c_p(n)\geq c(\log (n+1))^{\frac 2p}.
\]
\end{theorem}

\begin{proof}Fix a selfadjoint $M_n$-valued function $b.$ By the
operator Jensen inequality and Lemma \ref{holderforexpectation}, for
$p\geq 3,$

\begin{eqnarray*}
\|b\|_{\BMO}^2 &=&\sup_m\|\mathcal{E}_m|b-\mathcal{E}_{m-1}b|^2\|_\infty  \\
&\leq &\sup_m\|\mathcal{E}_m|b-\mathcal{E}_{m-1}b|^{\frac{p+1}2}\|_\infty ^{\frac 4{p+1}} \\
&\leq &\sup_m\|\mathcal{E}_m|b-\mathcal{E}_{m-1}b|^p\|_\infty
^{\frac
2{p+1}}\sup_m\|\mathcal{E}_m|b-\mathcal{E}_{m-1}b\||_\infty ^{\frac 2{p+1}} \\
&\leq
&(c_p(n)\|b\|_{\BMO})^{\frac{2p}{p+1}}\sup_m\|\mathcal{E}_m|b-\mathcal{E}_{m-1}b\||_\infty
^{\frac 2{p+1}}.
\end{eqnarray*}
Then
\[
\|b\|_{\BMO}\leq
(c_p(n))^p\sup_m\|\mathcal{E}_m|b-\mathcal{E}_{m-1}b\||_\infty .
\]
By Lemma \ref{negativeanswer1}, we get
\[
(c_p(n))^p\geq c(\log (n+1))^2.
\]
\end{proof}
From Theorem \ref{negativeanswer2}, we get a negative answer for the
open question by letting $n\rightarrow \infty .$

\vskip 0.2 in \noindent {\bf Acknowledgments.} The authors are
grateful to Quanhua Xu for his helpful discussions which led to a
big improvement of the paper.


\end{document}